\documentclass{article}
\usepackage[utf8]{inputenc}
\usepackage{amsmath}
\usepackage{amsthm}
\usepackage{amssymb}
\usepackage{scalerel}
\usepackage{mathtools}
\usepackage{hyperref}
\DeclarePairedDelimiter{\set}{\{}{\}}

\renewcommand{\phi}{\varphi}

\newcommand{\mc}{\mathcal}

\newcommand{\nat}{\ensuremath{\mathbb{N}}}
\newcommand{\zah}{\ensuremath{\mathbb{Z}}}

\newcommand{\rea}{\ensuremath{\mathbb{R}}}

\DeclareMathOperator{\Homeo}{Homeo}

\DeclareMathOperator{\MCG}{MCG}

\DeclareMathOperator*{\bigntrl}{\scalerel*{\natural}{\sum}}

\newtheorem{theorem}{Theorem}[section]
\newtheorem{defn}[theorem]{Definition}
\newtheorem{lemma}[theorem]{Lemma}
\newtheorem{cor}[theorem]{Corollary}
\newtheorem{fact}[theorem]{Fact}
\newtheorem{remark}[theorem]{Remark}

\newtheorem{prop}[theorem]{Proposition}

\newtheorem*{namedtheorem}{\theoremname}
\newcommand{\theoremname}{testing}
\newenvironment{named}[1]{\renewcommand{\theoremname}{#1}\begin{namedtheorem}}{\end{namedtheorem}}

\newcommand{\loopg}{\mc{L}}

\newcommand{\curveg}{\mc{C}}
\newcommand{\scg}{\mc{TC}}

\title{Graphs of curves and arcs quasi-isometric to big mapping class groups}
\author{Anschel Schaffer-Cohen}

\begin{document}

\maketitle

\begin{abstract}
	Following the work of Rosendal and Mann and Rafi, we try to answer the following question: when is the mapping class group of an infinite-type surface quasi-isometric to a graph whose vertices are curves on that surface? With the assumption of tameness as defined by Mann and Rafi, we describe a necessary and sufficient condition, called translatability, for a geometrically nontrivial big mapping class group to admit such a quasi-isometry. In addition, we show that the mapping class group of the plane minus a Cantor set is quasi-isometric to the loop graph defined by Bavard, which we believe represents the first known example of a hyperbolic mapping class group that is not virtually free.
\end{abstract}

\section{Introduction}

For our purposes, a \emph{surface} is a connected, oriented $2$-manifold without boundary. A surface is said to have \emph{finite type} if it is homeomorphic to a compact surface minus a finite set of points, or equivalently if its fundamental group is finitely generated. All other surfaces are said to have \emph{infinite type}. The \emph{mapping class group} of a surface $\Sigma$, which we write as $\MCG(\Sigma)$, is the group of orientation-preserving homeomorphisms of $\Sigma$ considered up to isotopy, sometimes written $\Homeo^+(\Sigma)/\Homeo_0(\Sigma)$ or $\pi_0(\Homeo^+(\Sigma))$. When $\Sigma$ has infinite type, we call $\MCG(\Sigma)$ \emph{big} to distinguish it from the better-studied mapping class groups of finite-type surfaces.

One of the primary goals of geometric group theory is to study a group via its actions on metric spaces. In the case of a finitely generated group, the word metric on the group gives a coarse upper bound on the displacement of a point under any action of the group on a metric space. If in addition the action is geometric, the word metric on the group also gives a coarse \emph{lower} bound on the displacement, making the orbit map a quasi-isometric embedding; the addition of coarse transitivity makes the orbit map a true quasi-isometry.

Even without a quasi-isometry, it is often possible to extract data about the geometry of the group from the geometry of the space on which it acts. Most famously in the case of mapping class groups, Masur and Minsky \cite{mm2} gave estimates for the word length of a mapping class based on its action on a sequence of curve graphs.

In the case of big mapping class groups this classical approach cannot be directly applied, because they are not finitely generated---indeed, they are uncountable. Big mapping class groups do, however, have a non-trivial topology inherited from $\Homeo^+(\Sigma)$, so if we restrict our attention to \emph{continuous} actions on metric spaces some of the old tools become available in a new light. Rosendal \cite{rosendal} provides the framework for this viewpoint in the context of general topological groups, using the notion of \emph{coarse boundedness}, which generalizes the concept of finiteness from the discrete case---see Definition \ref{cb_def}. For instance, instead of studying discrete groups that admit a finite generating set, we can study groups that admit a coarsely bounded neighborhood of the identity and a coarsely bounded generating set. Crucially, Fact \ref{qi_gen} below gives us a well-defined quasi-isometry class for such a group.

A recent paper of Mann and Rafi \cite{mr} provides a thorough application of this idea to the area of big mapping class groups. In particular, the paper provides (under the technical condition of \emph{tameness}---see Definition \ref{tame_defn}) a classification of which infinite-type surfaces admit coarsely bounded identity neighborhoods and generating sets; these generating sets are given explicitly. A natural question then follows: given such a big mapping class group, which thus has a well-defined quasi-isometry type, is there a ``nice'' metric space to which it is quasi-isometric?

There is one obvious candidate: the Cayley graph of the group, with respect to the generating set found by Mann and Rafi. This graph is simplicial, and is by construction quasi-isometric to the group, but it has two significant drawbacks. First, it has uncountably many vertices, which might limit the application of some combinatorial methods. Second, its a posteriori definition makes it unlikely to produce insights not available by direct examination of the group itself.

Instead, we turn to the wealth of already-defined simplicial graphs that admit a continuous action of a big mapping class group. Examples include Bavard's \emph{loop graph} \cite{bavard}, Rasmussen's \emph{nonseparating curve graph} \cite{rasmussen}, the graphs of separating curves defined by Durham, Fanoni, and Vlamis \cite{dfv}, and the general curve and arc graphs defined by Aramayona, Fossas, and Parlier \cite{afp}.\footnote{Fanoni, Ghaswala, and McLeay \cite{fgm} give an interesting action of a big mapping class group on the graph of \emph{omnipresent arcs}, but this action is not continuous when the graph is given the topology of a simplicial complex.}It should be noted that each of these graphs has as its vertices isotopy classes of some of the arcs or curves of the surface, and that actions on such a graph are always continuous. This lets us narrow our focus still further: given a big mapping class group with a well-defined quasi-isometry type, is there a simplicial graph whose vertices are arcs or curves on the underlying surface, such that the action of the group on the graph induces a quasi-isometry?

For graphs of curves we come to a very satisfying conclusion: we define a class of \emph{translatable} surfaces---essentially, surfaces admitting a map that acts with north-south dynamics with respect to two distinct ends (see Definition \ref{shifty_def})---and an associated \emph{translatable curve graph}, and show that the mapping class group of a translatable surface is quasi-isometric to its translatable curve graph. What's more, we show that non-translatable surfaces do not admit such a graph of curves, except when that graph has finite diameter. More precisely, we prove the following.

\begin{defn}\label{goc}
	A \emph{graph whose vertices are curves} on a surface $\Sigma$ is a simplicial graph whose vertex set $V$ is a subset of the set of isotopy classes of simple closed curves on $\Sigma$. If $\MCG(\Sigma)$ preserves both the set $V$ and the edge relation on the graph, we call the resulting group action \emph{the action of $\MCG(\Sigma)$ on the graph}. Note that such an action is always continuous.
\end{defn}

\begin{named}{Theorem \ref{only_shifty}}
	Let $\Sigma$ be an infinite-type surface with tame end space such that $\MCG(\Sigma)$ admits a coarsely bounded neighborhood of the identity and a coarsely bounded generating set---and thus has a well-defined quasi-isometry type---but is not itself coarsely bounded. Then the following are equivalent:
	\begin{enumerate}
		\item There exists a graph $\Gamma$ whose vertices are curves, with the action of $\MCG(\Sigma)$ on $\Gamma$ defined and inducing a quasi-isometry.
		\item $\Sigma$ is translatable.
		\item $\Sigma$ has no nondisplaceable finite-type surfaces, making it an \emph{avenue surface} in the sense of Horbez, Qing, and Rafi \cite{hqr}.
	\end{enumerate}
\end{named}

We do not attempt in this paper to study the geometry of the translatable curve graph, although we hope this will be a fruitful avenue for further research. One property is however immediate: by the results of Horbez, Qing, and Rafi \cite{hqr} the translatable curve graph---and thus the mapping class group of a translatable surface---cannot be non-elementary $\delta$-hyperbolic.

In the case of graphs of arcs we have not found such a general classification, but we exhibit one particularly striking quasi-isometry. Note that this surface is not translatable.

\begin{named}{Theorem \ref{cantor_qi}}
	Let $\Sigma = \rea^2 \setminus C$ be the plane minus a Cantor set. Then $\MCG(\Sigma)$ is quasi-isometric to the loop graph $\loopg(\Sigma)$.
\end{named}

Though less general than the previous result, this quasi-isometry is of interest because the loop graph is already well-studied; for instance, the hyperbolicity of the loop graph was demonstrated by Bavard \cite{bavard}, and thus $\MCG(\Sigma)$ is also hyperbolic; see Corollary \ref{hyp}. The Gromov boundary of this graph was also described by Bavard and Walker \cite{bw}.

This is, to our knowledge, the first case of a big mapping class group being shown to be non-elementary $\delta$-hyperbolic. By extension, it is also the first case in which two big mapping class groups have been shown to have distinct, non-trivial quasi-isometry types; see Corollary \ref{distinct}. Finally, it follows from Corollary \ref{hyp} and work of Cornulier and de la Harpe that $\MCG(\Sigma)$ has a coarsely bounded presentation; see Corollary \ref{cbp}.

Before getting to the meat of the paper, we introduce an important motivating example and preview some of the techniques that will be used.

\label{motivating}Durham, Fanoni, and Vlamis \cite{dfv}, studying the Jacob's ladder surface (which has two ends, both accumulated by genus---see Figure \ref{ladder}) present the following subgraph of the curve graph of that surface: its vertices are curves separating the two ends, with an edge between two such curves if they cobound a genus-one subsurface. The main immediate application of this graph results from the fact that, unlike the full curve graph of an infinite-type surface, it has infinite diameter. In particular, a translation acts on this graph with unbounded orbits, which provides an easy proof that the mapping class group of this surface is not coarsely bounded.

\begin{figure}
	\includegraphics[width=\textwidth]{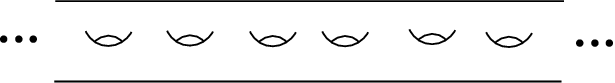}
	\caption{The Jacob's ladder surface.\label{ladder}}
\end{figure}

An early version of Vlamis's notes on the topology of big mapping class groups \cite{aim} claimed that this graph is quasi-isometric to the mapping class group of the Jacob's ladder surface. Vlamis's proof was incomplete---it showed only that the vertex stabilizers are coarsely bounded, which is not sufficient to conclude quasi-isometry---but it provided significant inspiration for the results which eventually became Theorem \ref{shifty_qi}.

First, this graph could in fact be shown to be quasi-isometric to the mapping class group, although it would take some additional effort. Second, the class of surfaces for which such a graph might be built could be expanded significantly beyond the Jacob's ladder surface. The properties of Durham, Fanoni, and Vlamis's graph depended largely on the translatable nature of the Jacob's ladder surface, rather than the details of the translation itself. Other surfaces admitting a similar kind of translation include the bi-infinite flute (see Figure \ref{flute}) and more complicated surfaces that might be built by joining many copies of a single surface as in Figure \ref{general_shiftiness} (on page \pageref{general_shiftiness}). By modifying the construction of Durham, Fanoni, and Vlamis \cite{dfv} we are able to produce a general \emph{translatable curve graph} $\scg(\Sigma)$ which is quasi-isometric to the mapping class group $\MCG(\Sigma)$; this is Theorem \ref{shifty_qi}.

\begin{figure}
	\includegraphics[width=\textwidth]{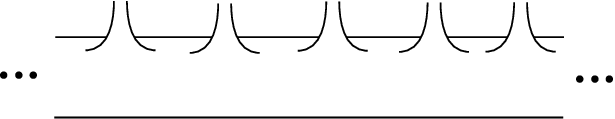}
	\caption{The bi-infinite flute.\label{flute}}
\end{figure}

One obvious follow-up question, given this quasi-isometry, is whether other such graphs can be produced. For instance, are there other cases where a big mapping class group is quasi-isometric to a graph whose vertices are curves? What if the vertices are arcs? Theorem \ref{only_shifty} answers the first question; Theorem \ref{cantor_qi} is a partial answer to the second.

The structure of the rest of this paper is as follows. In Section \ref{prelim}, we recall relevant results from previous papers (\cite{richards}, \cite{rosendal}, and \cite{mr}) that are used in this work, with the goal of making our work accessible to anyone with some knowledge of geometric group theory and low-dimensional topology, but who may not have worked previously with infinite-type surfaces or with the concept of coarse boundedness. Section \ref{prelim} also presents and proves Lemma \ref{ms}, which is a limited version of the Schwarz-Milnor lemma for the case of groups with coarsely bounded neighborhoods of the identity acting transitively on graphs.

In Section \ref{shifty_sec}, we define translatable surfaces and prove some of their properties, most notably Proposition \ref{segments}, which shows that every translatable surface can be written as an infinite connected sum of copies of some subsurface $S$ as in Figure \ref{general_shiftiness}.

In Section \ref{graph_sec} we define the translatable curve graph itself and prove the quasi-isometry to the mapping class group in Theorem \ref{shifty_qi}. The main tools are a study of the maximal ends of the subsurface $S$ found in Proposition \ref{segments}, and Lemma \ref{shift_enough}, which allows us to embed the set of all mapping classes that fix half of our translatable surface in a conjugate of any neighborhood of the identity.

In Section \ref{other_graphs} we show that translatable surfaces are in fact the only surfaces with non-coarsely-bounded mapping class groups quasi-isometric to a graph of curves, proving Theorem \ref{only_shifty}. The main tools here are, on one hand, a demonstration that under some reasonable conditions any surface with two equivalent maximal ends and zero or infinite genus is translatable; and on the other hand, that all other surfaces have mapping class groups that are either themselves coarsely bounded or have no coarsely bounded curve stabilizers, making such a quasi-isometry impossible.

Finally, Section \ref{cantor_sec} uses methods parallel to those in Sections \ref{shifty_sec} and \ref{graph_sec} to prove that the mapping class group of the plane minus a Cantor set is quasi-isometric to the loop graph of that surface.

\subsection*{Acknowledgements}

I am indebted to Nicholas G.\ Vlamis for giving the seminar talk that initially turned me on to the study of big mapping class groups, and for subsequent advice and direction, including introducing me to coarse boundedness. Rylee Lyman suggested the $\natural$ notation used in Proposition \ref{segments}. Kathryn Mann and Kasra Rafi, in addition to producing the seminal results that make this work possible, have also been patient and helpful in responding to my follow-up questions about their paper. I have also received helpful feedback from Federica Fanoni and of course from my advisor, Dave Futer, who has also slogged through all of my drafts. Some of this work was done at the Young Geometric Group Theory IX conference, which was partially funded by NSF grant 1941077.

\section{Preliminaries}\label{prelim}

Before we begin, we recall several results from the work of Rosendal \cite{rosendal} and Mann and Rafi \cite{mr}, as well as the classification of infinite-type surfaces due to Kerékjártó \cite{kerekjarto} and Richards \cite{richards}.

\subsection{Coarse boundedness}

\begin{defn}\label{cb_def}
	A subset $A$ of a topological group $G$ is \emph{coarsely bounded} if it has finite diameter in every left-invariant compatible pseudo-metric on $G$.
\end{defn}

\begin{defn}
	We call a group \emph{locally coarsely bounded} if it has a coarsely bounded neighborhood of the identity.
\end{defn}

For our purposes, the value of this definition lies precisely in the following result:

\begin{fact}[From \cite{rosendal}]\label{qi_gen}
	If $A$ and $B$ are two coarsely bounded generating sets for a locally coarsely bounded group $G$, then the word metrics with respect to the generating sets $A$ and $B$ are quasi-isometric.
\end{fact}

We make heavy use of the following alternate characterization of coarse boundedness:

\begin{fact}[From \cite{rosendal}]\label{cb_nbhd}
	Given $G$ a Polish group, and a subset $A \subseteq G$. Then $A$ is coarsely bounded if and only if for every identity neighborhood $V \subseteq G$ there is some $k \in \nat$ and a finite set $F \subseteq G$ such that $A \subseteq (FV)^k$.
\end{fact}

In light of this definition, we will want to talk about specific identity neighborhoods in the mapping class group of a surface $\Sigma$: if $S$ is a subsurface of $\Sigma$, let $V_S$ be the set of mapping classes with representatives that restrict to the identity on $S$. Note that the set $\set{V_S \mid \text{$S \subseteq \Sigma$ of finite type}}$ forms a neighborhood basis of the identity in $\MCG(\Sigma)$.

We're looking to prove quasi-isometries between groups and graphs, so we want something reminiscent of the Schwarz-Milnor lemma. Fact \ref{qi_gen} makes our work much easier.

\begin{lemma}\label{ms}
	Let $G$ be a locally coarsely bounded group acting transitively by isometries on a connected graph $\Gamma$ equipped with the edge metric. Suppose that for some vertex $v_0 \in \Gamma$, the set $A = \set{g \in G \mid d(v_0, gv_0) \le 1}$ is coarsely bounded. Then the orbit map $g \mapsto g v_0$ is a quasi-isometry.
\end{lemma}
\begin{proof}
	Coarse surjectivity follows directly from the transitivity of the action.

	Fix some $g \in G$. Since $\Gamma$ is connected, there is a minimal-length path $v_0, v_1, \dotsc, v_n = gv_0$ from $v_0$ to $gv_0$ with $d(v_i, v_{i+1}) = 1$. Since $d(v_0, v_1) = 1$ and the action of $G$ is transitive, there is some $g_0 \in A$ such that $g_0 v_0 = v_1$. Likewise, there is some $g_1' \in g_0 A g_0^{-1}$ such that $g_1' v_1 = v_2$. Writing $g_1' = g_0 g_1 g_0^{-1}$ with $g_1 \in A$ and $v_1 = g_0 v_0$, we see that $g_0 g_1 v_0 = v_2$. Continuing in this way, we can find $g_0, g_1, \dotsc, g_{n-1} \in A$ such that $g_0 g_1 \dotsm g_{n-1} v_0 = v_n$. Let $g_n = g^{-1} g_0 g_1 \dotsm g_{n-1}$. Then $g_n v_0 = v_0$, so $g_n \in A$, and $g_0 g_1 \dotsm g_{n-1} g_n = g$. Thus $A$ is a generating set for $G$ and the word-metric length of $g$ in $A$ is at most $n = d(v_0, g v_0) + 1$.

	On the other hand, suppose $g_0 g_1 \dotsm g_k = g$ with each $g_i \in A$ and $k$ minimal. By the definition of $A$, $d(g_0 g_1 \dotsm g_i v_0, g_0 g_1 \dotsm g_{i+1} v_0) = d(v_0, g_{i+1} v_0) \le 1$, so the distance $d(v_0, g v_0) \le k+1$. Thus the map $g \mapsto g v_0$ coarsely preserves the word metric with the generating set $A$, and thus by Fact \ref{qi_gen} the orbit map is a quasi-isometry for any choice of coarsely bounded generating set for $G$.
\end{proof}

Before we can apply this lemma, we need to know a bit more about infinite-type surfaces.

\subsection{Infinite-type surfaces}

The classification of infinite-type surfaces was first given by Kerékjártó \cite{kerekjarto}, whose proof was corrected by Richards \cite{richards}. It is based on the following notion of \emph{ends}:

\begin{defn}
	Given a surface $\Sigma$, an \emph{end} of $\Sigma$ is a nested sequence of subsurfaces $S_1 \supseteq S_2 \supseteq \dotsb$ of $\Sigma$, each with compact boundary and with the property that for any compact subsurface $K \subseteq \Sigma$, $K \cap S_n = \emptyset$ for high enough $n$. Two such sequences $S_1 \supseteq S_2 \supseteq \dotsb$ and $T_1 \supseteq T_2 \supseteq \dotsb$ are considered to be the same end if for every $n \in \nat$ there exists $m \in \nat$ such that $T_m \subseteq S_n$, and for every $n \in \nat$ there exists an $m \in \nat$ such that $S_m \subseteq T_n$.

	An end $x$ given by a sequence $S_1 \supseteq S_2 \supseteq \dotsb$ is said to be \emph{accumulated by genus} if every $S_n$ has positive genus. Otherwise $x$ is said to be \emph{planar}.

	The \emph{space of ends} of $\Sigma$, written $E(\Sigma)$, is a topological space whose points are the ends of $\Sigma$ and whose basic open sets correspond to subsurfaces $S \subseteq \Sigma$ with compact boundary. An end $S_1 \supseteq S_2 \supseteq \dotsb$ of $\Sigma$ is contained in the basic open set corresponding to $S$ if $S_n \subseteq S$ for high enough $n$, and this basic open set is a neighborhood of this end.
\end{defn}

The topological space $E(\Sigma)$ is compact, second-countable, and totally separated; this last condition means that for every pair of ends $x, y \in E$, there is a clopen subset $U \subseteq E$ containing $x$ but not $y$. In fact, every separating curve on $\Sigma$ divides $E$ into two clopen subsets---one of which may be empty---and the sets so defined form a countable basis of $E$. The set of ends accumulated by genus is written $E_G$ and is a closed subset of $E$.

\begin{fact}[The principal result of \cite{richards}]
	An orientable surface without boundary is classified by its genus, which may be infinite, its space of ends $E$, and the subset $E_G \subseteq E$ of ends accumulated by genus. What's more, any homeomorphism of the pair $(E, E_G)$ extends to a homeomorphism of the underlying surface.
\end{fact}

Following the example of Mann and Rafi, we will mostly avoid referencing $E_G$ explicitly, and implicitly assume it is preserved. For instance, when we say that two subsets $U, V \subseteq E$ are homeomorphic, we mean that there is a homeomorphism $f \colon U \to V$ such that $f(U \cap E_G) = V \cap E_G$.

Mann and Rafi introduce the following partial pre-order on $E$, which provides a valuable for tool for studying its topology:
\begin{defn}
	Given $x, y \in E$, we say $x \succcurlyeq y$ if for every clopen neighborhood $U$ of $x$, there exists a clopen neighborhood $V$ of $y$ homeomorphic to a clopen subset of $U$.
\end{defn}

As might be expected, we write $y \prec x$ when $y \preccurlyeq x$ but $x \not\preccurlyeq y$, and $y \sim x$ when $y \preccurlyeq x$ and $x \preccurlyeq y$. We use the notation $E(x) = \set{y \in E \mid x \sim y}$ for the equivalence class of $x \in E$ under the relation $\sim$.

Crucially, this order plays well with the topology of the space of ends and has maximal elements, which have a fairly rigid structure.

\begin{fact}[Lemma 4.6 of \cite{mr}]\label{bigger_closed}
	For every $y \in E$, the set $\set{x \in E \mid x \succcurlyeq y}$ is closed.
\end{fact}

\begin{fact}[Proposition 4.7 of \cite{mr}]\label{max_exist}
	The partial pre-order $\preccurlyeq$ has maximal elements. Furthermore, for every maximal element $x \in E$, the equivalence class $E(x)$ is either finite or a Cantor set.
\end{fact}

The following lemma is a slight generalization of the first part of Fact \ref{max_exist}, following precisely the same proof.

\begin{lemma}\label{closed_max_exist}
	Every nonempty closed subset $F \subseteq E$ has maximal elements.
\end{lemma}
\begin{proof}
	Suppose $\mc{C}$ is a totally ordered subset (a chain) in $F$. For each $y \in \mc{C}$, the set $\set{x \in F \mid x \succcurlyeq y}$ is closed by Fact \ref{bigger_closed} and thus compact because $E$ itself is compact, and is non-empty by construction. Since $\mc{C}$ is totally ordered, these sets are nested, and so their intersection $\bigcap_{y \in \mc{C}} \set{x \in F \mid x \succcurlyeq y}$ is nonempty and contains a maximal element of the chain $\mc{C}$. Thus by Zorn's lemma the set $F$ has a maximal element.
\end{proof}

Mann and Rafi also define the following self-similarity condition, and prove some useful properties of it:

\begin{defn}
	A clopen neighborhood $U$ of an end $x \in E$ is \emph{stable} if for every clopen neighborhood $U'$ of $x$ contained in $U$, there is a clopen subset of $U'$ homeomorphic to $U$.
\end{defn}

\begin{fact}[Lemma 4.17 of \cite{mr}]\label{stable_homeo}
	If $x \sim y \in E$ and $x$ has a stable neighborhood $U$, then all sufficiently small neighborhoods of $y$ are homeomorphic to $U$ via a homeomorphism taking $x$ to $y$.
\end{fact}

\begin{fact}[Lemma 4.18 of \cite{mr}]\label{add_in}
	Let $x, y \in E$ and assume $x$ has a stable neighborhood $V_x$, and that $x$ is an accumulation point of $E(y)$. Then for any sufficiently small clopen neighborhood $U$ of $y$, $U \cup V_x$ is homeomorphic to $V_x$.
\end{fact}

Many of the results in later sections will assume the existence of certain stable neighborhoods, in the form of what Mann and Rafi call \emph{tameness}:

\begin{defn}\label{tame_defn}
	An end space $E$ is said to be \emph{tame} if any $x \in E$ that is either maximal or an immediate predecessor to a maximal end has a stable neighborhood.
\end{defn}

It is an open question (Problem 6.15 of \cite{mr}) whether there exists any surface with non-tame end space whose mapping class group is not coarsely bounded but has a well-defined quasi-isometry type. For this reason, we consider tameness to be an acceptable condition to impose in some of our results.

To achieve the negative results of subsection \ref{neg_results}, we will need to consider the properties of mapping class groups that are locally coarsely bounded and admit a coarsely bounded generating set. Here $\mc{M}(X)$ denotes the set of maximal ends of some subspace $X \subseteq E$.

\begin{fact}[Theorem 1.4 of \cite{mr}]\label{cb_class}
	$\MCG(\Sigma)$ is locally coarsely bounded if and only if either $\MCG(\Sigma)$ is itself coarsely bounded or there is a finite-type surface $K \subseteq \Sigma$ such that the complementary regions of $K$ each have infinite type and zero or infinite genus, and partition $E$ into finitely many clopen sets
	\[ E = \left( \bigsqcup_{A \in \mc{A}} A \right) \sqcup \left( \bigsqcup_{P \in \mc{P}} P \right) \]
	such that:
	\begin{enumerate}
		\item Each $A \in \mc{A}$ is self-similar, with $\mc{M}(A) \subseteq \mc{M}(E)$ and $\mc{M}(E) \subseteq \bigsqcup_{A \in \mc{A}} \mc{M}(A)$,
		\item each $P \in \mc{P}$ is homeomorphic to a clopen subset of some $A \in \mc{A}$, and
		\item for any $x_A \in \mc{M}(A)$, and any neighborhood $V$ of the end $x_A$ in $\Sigma$, there is $f_V \in \MCG(\Sigma)$ so that $f_V(V)$ contains the complementary component to $K$ with end space $A$.
	\end{enumerate}
	Moreover, in this case $V_K$---the set of mapping classes restricting to the identity on $K$---is a coarsely bounded neighborhood of the identity.
\end{fact}

We will also make use of the following necessary condition for $\MCG(\Sigma)$ to have a coarsely bounded generating set:

\begin{defn}[Definition 6.2 of \cite{mr}]
	We say that an end space $E$ has \emph{limit type} if there is a finite-index subgroup $G$ of $\MCG(\Sigma)$, a $G$-invariant set $X \subseteq E$, points $z_n \in E$ indexed by $n \in \nat$ which are pairwise inequivalent, and a nested family of clopen sets $U_n$ with $\bigcap_{n \in \nat} U_n = X$ such that
	\[ E(z_n) \cap U_n \ne \emptyset, \quad E(z_n) \cap U_0^c \ne \emptyset, \quad \text{and} \quad E(z_n) \subseteq (U_n \cup U_0^c) \]
	where $U_0^c = E \setminus U_0$.
\end{defn}

This definition is somewhat daunting, so we present the following example of a surface whose end space has limit type. Let $z_1$ be a puncture, and for each $n > 1$ let $z_n$ be an end accumulated by countably many points locally homeomorphic to $z_{n-1}$. Then let $z_\omega$ be an end accumulated by $\set{z_n}_{n \in \nat}$. Let $F$ be the set of ends just defined, and let $\Sigma$ be a surface with zero genus and end space homeomorphic to the disjoint union of $n$ copies of $F$ for any $n > 1$. It can be verified that the end space of $\Sigma$ has limit type.

\begin{fact}[Lemma 6.4 of \cite{mr}]\label{limit_type}
	If $E$ has limit type, then $\MCG(\Sigma)$ does not admit a coarsely bounded generating set.
\end{fact}

\section{Translations on surfaces}\label{shifty_sec}

We first define a useful notion of convergence:

\begin{defn}
	Given a surface $\Sigma$, an end $e$ of $\Sigma$, and a sequence $\alpha_1, \alpha_2, \dotsc$ of curves on $\Sigma$, we say that $\lim_{n \to \infty} \alpha_n = e$ if, for every neighborhood $V$ of the end $e$ in the surface $\Sigma$, all but finitely many of the $\alpha_i$ are (after some isotopy) contained in $V$.
\end{defn}

A translation, then, will be a map that moves all curves toward one end and away from another. That is:

\begin{defn}\label{shifty_def}
	Given a surface $\Sigma$, a map $h \in \MCG(\Sigma)$ is called a \emph{translation} if there are two distinct ends $e_+$ and $e_-$ of $\Sigma$ such that for any curve $\alpha$ on $\Sigma$, $\lim_{n \to \infty} h^n(\alpha) = e_+$ and $\lim_{n \to \infty} h^{-n}(\alpha) = e_-$. If such a translation exists, we call the surface $\Sigma$ \emph{translatable}.
\end{defn}

\begin{remark}
	This definition brings to mind several other classes of infinite-type surfaces with two special ends that have recently been defined for various reasons.

	\begin{itemize}
		\item The \emph{telescoping} surfaces of Mann and Rafi \cite{mr} form a strict subset of the translatable surfaces: though every telescoping surface can be shown to be translatable, the Jacob's ladder surface is translatable but not telescoping.
		\item The \emph{doubly pointed} surfaces of Aougab, Patel, and Vlamis \cite{apv} are a strict superset of the translatable surfaces whose mapping class groups are not coarsely bounded: every such translatable surface is doubly pointed, but the surface with zero genus and end space homeomorphic to $\omega^\omega 2 + 1$ is doubly pointed but not translatable.\footnote{There are easier counterexamples, e.g.\ a surface with two inequivalent maximal ends (see Lemma \ref{two_diff_ends}). This example demonstrates however that a doubly pointed surface may not be translatable even if it has exactly two equivalent maximal ends.}
		\item The \emph{avenue} surfaces of Horbez, Qing, and Rafi \cite{hqr} turn out to be precisely those translatable surfaces that have tame end space and whose mapping class groups are not coarsely bounded. This result is part of Theorem \ref{only_shifty}.
	\end{itemize}
\end{remark}

It follows directly from the definition that a translatable surface $\Sigma$ cannot contain any finite-type nondisplaceable surfaces, and so in particular $\Sigma$ cannot have finite type, finite positive genus, or any ends with a finite $MCG(\Sigma)$-orbit of size more than $2$. This gives us lots of non-examples, but there are also plenty of translatable surfaces if we go looking for them.

The Jacob's ladder surface and the bi-infinite flute, in Figures \ref{ladder} and \ref{flute}, are clearly translatable; we can think of the flute as being derived from the ladder by replacing each handle with a puncture. More generally, we might replace each handle in the ladder with some other surface, as follows. Let $S$ be any surface, not necessarily of finite type, with two compact boundary components, and let $\Sigma = S^{\natural \zah}$ be the gluing along their boundaries\footnote{The use of $\natural$ here is intended to invoke the standard use of $\#$ for connected sum, and was suggested to me on Facebook by Rylee Lyman.} of countably many copies of $S$, arranged like \zah\ as in Figure \ref{general_shiftiness}. Then the map that takes each copy of $S$ to the next one over is a translation, and so $\Sigma$ is translatable.

\begin{figure}
	\includegraphics[width=\textwidth]{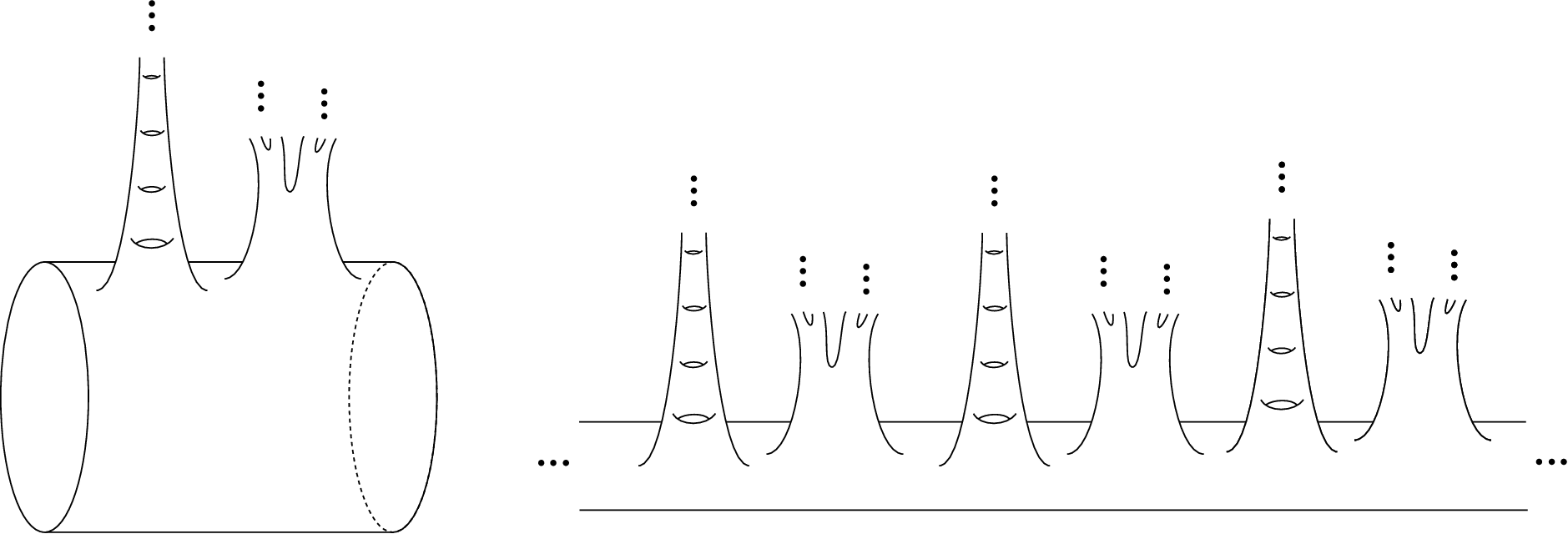}
	\caption{On the left, a surface $S$ with two boundary components (in this case, the connected sum of an annulus, a Loch Ness Monster, and a Cantor tree). On the right, the translatable surface $S^{\natural \zah}$. \label{general_shiftiness}}
\end{figure}

A natural question to ask is whether this last example includes all translatable surfaces, and in fact it does. However, we may have to choose a different translation. For this and future results, the following notation will be useful:

\begin{defn}
	Suppose $\Sigma$ is a translatable surface and $\alpha$ a curve in $\Sigma$ separating $e_+$ and $e_-$. Then we denote by $\alpha_+$ (resp.\ $\alpha_-$) the component of $\Sigma \setminus \alpha$ containing the end $e_+$ (resp.\ $e_-$). If $\beta$ is a curve separating $\alpha$ and $e_+$, then we denote by $(\alpha, \beta)$ the subsurface $\alpha_+ \cap \beta_-$ of $\Sigma$ bounded by $\alpha$ and $\beta$.
\end{defn}

\begin{prop}\label{segments}
	Let $\Sigma$ be a translatable surface with translation $h$ and $\alpha$ a curve separating the ends $e_+$ and $e_-$. Then there is a surface $S = (\alpha, h^N(\alpha))$ for some $N$ such that $\Sigma$ is homeomorphic to $S^{\natural \zah}$.
\end{prop}
\begin{proof}
	By the definition of translation, there is some $N \in \nat$ such that for all $n \ge N$, $h^n(\alpha) \subseteq \alpha_+$. Then $h^N$ is also a translation, so without loss of generality we can replace $h$ with $h^N$ and assume that all $h^i(\alpha)$ are disjoint.

	Let $S = (\alpha, h(\alpha))$. If $x \in \Sigma$ but not in any $h^i(\alpha)$, then $x$ is either in $\alpha_+$ or $\alpha_-$. Supposing without loss of generality that $x \in \alpha_+$, there must be some least $i$ such that $x \not\in h^i(\alpha)_+$, otherwise $h^i(\alpha)$ could not converge to $e_+$. Then $x \in (h^{i-1}(\alpha), h^i(\alpha)) = h^{i-1}(S)$. On the other hand, $h^i(S) \cap S = \emptyset$ for all $i \ne 0$ by construction, and so every point of $\Sigma$ is in exactly one $h^i(S)$ or $h^i(\alpha)$.

	The resulting subsurface $S$ has  two boundary components, and the copies of $S$ are glued together exactly as desired.
\end{proof}

This decomposition depended on the choice of a curve separating $e_+$ and $e_-$. We might ask how important that choice was, and it turns out the answer is ``not much''.

\begin{lemma}\label{transitive}
	Let $\Sigma$ be a translatable surface, and $\alpha$ and $\beta$ two curves separating the ends $e_+$ and $e_-$. Then there is some $f \in \MCG(\Sigma)$ which fixes the ends $e_+$ and $e_-$ and such that $f(\alpha) = \beta$.
\end{lemma}
\begin{proof}
	First, replace $\beta$ with some $h^n(\beta)$ so that $\beta \subseteq \alpha_+$, and then replace $h$ with some power of $h$ so that $\beta \subseteq (\alpha, h(\alpha))$. By Proposition \ref{segments}, we can write $\Sigma = S^{\natural \zah} = T^{\natural \zah}$, where $S = (\alpha, h(\alpha))$ and $T = (\beta, h(\beta))$.

	If we let $X = (\alpha, \beta)$, $Y = (\beta, h(\alpha))$, and $Z = (h(\alpha), h(\beta))$, then $(\alpha, h(\alpha)) = X \natural Y$ and $(\beta, h(\beta)) = Y \natural Z$, where by $A \natural B$ we mean the surface obtained by gluing the surfaces $A$ and $B$ along a single boundary component. But $X$ and $Z$ are homeomorphic, and thus so are $S = (\alpha, h(\alpha))$ and $T = (\beta, h(\beta))$. It follows that we can map each copy of $S$ to the appropriate copy of $T$, giving us a homeomorphism of $\Sigma$ that takes $\alpha$ to $\beta$ and fixes the ends $e_+$ and $e_-$.
\end{proof}

There is one more symmetry of a translatable surface worth discussing here:

\begin{lemma}\label{rotation}
	Let $\Sigma$ be a translatable surface, and $\alpha$ a curve separating the ends $e_+$ and $e_-$. Then there is some $r \in \MCG(\Sigma)$ that transposes $e_+$ and $e_-$ and restricts to an orientation-reversing homeomorphism on $\alpha$.
\end{lemma}
\begin{proof}
	In a tubular neighborhood of $\alpha$, $r$ is just a rotation by $\pi$ about a diameter of the circle $\alpha$---see Figure \ref{rot_fig}. By Proposition \ref{segments}, $\alpha_+$ and $\alpha_-$ are homeomorphic, so this $r$ can be extended to a homeomorphism on all of $\Sigma$.
\end{proof}

\begin{figure}
	\center\includegraphics[width=\textwidth/2]{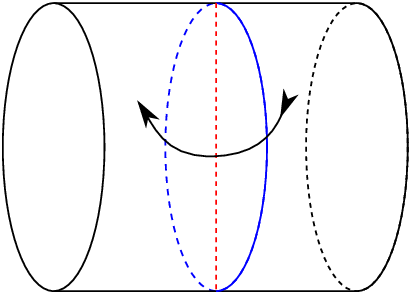}
	\caption{Within a tubular neighborhood of the blue curve $\alpha$, the curve's orientation can be reversed via a rotation by $\pi$ about its red diameter. \label{rot_fig}}
\end{figure}

The following lemma will be quite useful in light of Fact \ref{cb_nbhd}.

\begin{lemma}\label{shift_enough}
	Let $\Sigma$ be a translatable surface with translation $h$, and $\alpha$ a curve separating the ends $e_+$ and $e_-$. Then for any identity neighborhood $V$ in $\MCG(\Sigma)$, there is some $n \in \nat$ such that $V_{\alpha_-} \subseteq h^{-n} V h^n$ and $V_{\alpha_+} \subseteq h^n V h^{-n}$.
\end{lemma}
\begin{proof}
	By the topology of $\MCG(\Sigma)$, there is some finite-type subsurface $T \subseteq \Sigma$ such that $V_T \subseteq V$. Let $n \in \nat$ so that $T \subseteq h^n(\alpha)_-$. Then we have
	\begin{align*}
		T &\subseteq h^n(\alpha)_- \\
		V_{h^n(\alpha)_-} &\subseteq V_T \\
		V_{h^n(\alpha)_-} &\subseteq V \\
		h^n V_{\alpha_-} h^{-n} &\subseteq V \\
		V_{\alpha_-} &\subseteq h^{-n} V h^n \\
	\end{align*}
	and likewise $V_{\alpha_+} \subseteq h^n V h^{-n}$.
\end{proof}

\begin{cor}\label{stabilizer}
	Let $\Sigma$ be a translatable surface with translation $h$, and $\alpha$ a curve separating the ends $e_+$ and $e_-$. Then the set $H = \set{f \in \MCG(\Sigma) \mid \text{$f(\alpha)$ is homotopic to $\alpha$}}$ of mapping classes stabilizing $\alpha$ is coarsely bounded.
\end{cor}
\begin{proof}
	Fix an identity neighborhood $V$, and using Lemma \ref{shift_enough} find $n \in \nat$ so that $V_{\alpha_-} \subseteq h^{-n} V h^n$ and $V_{\alpha_+} \subseteq h^n V h^{-n}$. Let $F = \set{r^{-1}, h^n, h^{-n}}$ where $r$ is the map defined in Lemma \ref{rotation}. We claim that $H \subseteq (FV)^5$, which gives the result by Fact $\ref{cb_nbhd}$.

	Pick $f \in H$. Up to homotopy, $f_{|\alpha}$ is either the identity or a reflection map; in the latter case replace $f$ with $rf$ so that $f_{|\alpha}$ is the identity. Then the action of $f$ on $\alpha_+$ and $\alpha_-$ do not interact, and so $f$ can be decomposed as $f = f_-f_+$, where $f_- \in V_{\alpha_-}$ and $f_+ \in V_{\alpha_+}$. It follows by our choice of $n$ that $f_- \in h^{-n} V h^n$ and $f_+ \in h^n V h^{-n}$, so $f \in h^{-n} V h^n h^n V h^{-n} \subseteq (FV)^4$. Since we may also have applied $r^{-1}$, this gives $f \in (FV)^5$.
\end{proof}
\begin{cor}\label{local_cb}
	Let $\Sigma$ be a translatable surface. Then $\MCG(\Sigma)$ is locally coarsely bounded.
\end{cor}
\begin{proof}
	The stabilizer $H$ of a curve separating the ends $e_+$ and $e_-$ is an identity neighborhood, and by Corollary \ref{stabilizer} it is coarsely bounded.
\end{proof}

Note that unlike in the following sections, here we have not assumed tameness.

\section{The translatable curve graph}\label{graph_sec}

We are now ready to define a graph quasi-isometric to $\MCG(\Sigma)$ when $\Sigma$ is a translatable surface.

\begin{defn}
	Fix a collection $\mc{S}$ of subsurfaces of $\Sigma$. The \emph{translatable curve graph} $\scg(\Sigma, \mc{S})$ of $\Sigma$ with respect to the set of subsurfaces $\mc{S}$ is the graph whose vertices are curves separating $e_+$ and $e_-$, with an edge between two curves $\alpha$ and $\beta$ if they have disjoint representatives and $(\alpha, \beta)$ or $(\beta, \alpha)$ is homeomorphic to some $S \in \mc{S}$.
\end{defn}

We will eventually define a canonical and finite set $\mc{S}$ depending only on the surface $\Sigma$; once this has been defined, we will omit $\mc{S}$ and write simply $\scg(\Sigma)$.

Note that Corollary \ref{stabilizer} implies that vertex stabilizers of $\scg(\Sigma, \mc{S})$ are coarsely bounded. Also, with $\Sigma$ the Jacob's ladder surface and $S$ a surface with genus $1$ and two boundary components, the graph $\scg(\Sigma, \set{S})$ is precisely the motivating example from page \pageref{motivating}.

For an arbitrary translatable surface $\Sigma$, we might try taking a subsurface $S$ such that $\Sigma = S^{\natural \zah}$ as in Proposition \ref{segments}. But $\scg(\Sigma, \set{S})$ will not in general be connected. Instead, we will use the topology of $S$ to construct a collection of subsurfaces $\mc{S}$ such that $\scg(\Sigma, \mc{S})$ satisfies the conditions of Lemma \ref{ms}.

Consider the space of ends $E(S)$ of $S$; we may ask how these relate to the order structure on $E(\Sigma)$.

\begin{lemma}
	The maximal ends of the subsurface $S$ are either maximal ends of $\Sigma$ or immediate predecessors to maximal ends of $\Sigma$.
\end{lemma}
\begin{proof}
	It follows from the decomposition given in Proposition \ref{segments} that $e_+$ and $e_-$ are maximal ends of $\Sigma$, and that $e_+ \sim e_-$. In fact, they are global maxima of the partial preorder: for any end $x$ of $\Sigma$, $x \preccurlyeq e_+$. If $E(e_+)$ contains some point $y$ distinct from $e_+$ and $e_-$, then $y$ is in the end space of some copy of $S$. In particular, since it is still true that for every end $x$ of $\Sigma$, $x \preccurlyeq e_+ \sim y$, $E(e_+)$ contains \emph{all} the maximal ends of $S$.\footnote{This implies, by Proposition 4.8 of \cite{mr}, that the end space of $\Sigma$ is self-similar and thus $\MCG(\Sigma)$ is coarsely bounded, and we will in fact see that $\scg(\Sigma, \mc{S})$ has finite diameter under these circumstances.}

	If, on the other hand, $E(e_+) = \set{e_+, e_-}$, then let $x$ be maximal in $S$. We know that $x \prec e_+$. Suppose we have an end $y$ such that $x \preccurlyeq y \preccurlyeq e_+$. If $y$ is an end of some copy of $S$, then by maximality $y \sim x$. But if $y$ is not an end of any copy of $S$, the only other possibility is that $y = e_\pm$, in which case $y \sim e_+$. Thus $x$ is an immediate predecessor to $e_+$.
\end{proof}

\begin{cor}
	If the end space of $\Sigma$ is tame, then every maximal end of $S$ has a stable neighborhood.
\end{cor}

In general, a surface might have infinitely many equivalence classes of maximal ends. With tameness, however, the possibilities are much more limited.

\begin{lemma}\label{finite_max}
	If every maximal end of $T \subseteq \Sigma$ has a stable neighborhood, then the end space $E(T)$ has finitely many equivalence classes of maximal ends.
\end{lemma}
\begin{proof}
	For each maximal end $x$ of $T$, let $V_x$ be a stable neighborhood of $x$. For every end $y$, pick a clopen neighborhood $U_y$ of $y$ such that $U_y$ is homeomorphic to a clopen subset of $V_x$ for some maximal end $x$.

	Since the neighborhoods $U_y$ cover the end space $E(T)$, and $E(T)$ is compact, there is a finite set $U_1, \dotsc, U_n$ covering $E(T)$, where each $U_i$ is homeomorphic to a clopen subset of $V_{x_i}$.

	Now for each $y \in E(T)$, $y \in U_i$ for some $i$. Let $V$ be an arbitrary clopen neighborhood of $x_i$; by stability, $V$ contains a homeomorphic copy of $V_{x_i}$, which in turn contains a homeomorphic copy of $U_i$ by construction. Thus $y \preccurlyeq x_i$. It follows that the set $\set{x_1, \dotsc, x_n}$ contains a representative of every equivalence class of maximal ends, so there are at most $n$ such equivalence classes.
\end{proof}

While there are finitely many equivalence classes of maximal ends, a priori the non-maximal ends might contribute meaningfully to the topology of a subsurface. However, this is not the case.

\begin{lemma}\label{just_stable}
	Given a subsurface $T \subseteq \Sigma$ with end space $E(T)$ and such that every maximal end of $T$ has a stable neighborhood, $E(T)$ can be written as the disjoint union $\bigsqcup_{i=1}^k V_i$ where each $V_i$ is a stable neighborhood of a maximal end of $T$.
\end{lemma}
\begin{proof}
	Let $\mc{M}(T)$ be the set of maximal ends of $E(T)$. We start by constructing a subset $M \subseteq \mc{M}(T)$ such that $M$ contains every $x \in \mc{M}(T)$ where $\mc{M}(T) \cap E(x)$ is finite, and one representative of $\mc{M}(T) \cap E(x)$ when this intersection is infinite. By Lemma \ref{finite_max} $M$ is finite and thus discrete, so we can pick disjoint stable neighborhoods $V_x \subseteq E(T)$ for each $x \in M$. Let $V = \bigsqcup_{x \in M} V_x$.

	For each $y \in E(T) \setminus V$, there is by maximality some $x \in M$ such that $y \preccurlyeq x$. If $y \prec x$, then $x$ is an accumulation point of $E(y)$; and if $y \sim x$, then since $y \not\in V$ the set $E(x) \cap E(T)$ must be infinite, and thus a Cantor set by Fact \ref{max_exist}, and so again $x$ is an accumulation point of $E(y) = E(x)$. Thus in either case we can apply Fact \ref{add_in} to find some clopen $U_y \ni y$ such that $U_y \subseteq E(T) \setminus V$ and $U_y \cup V_x$ is homeomorphic to $V_x$. The set $E(T) \setminus V$ is clopen and thus compact, and is covered by the neighborhoods $U_y$, so there is a finite set of these neighborhoods covering $E(T) \setminus V$; since they are all clopen we can ensure they are disjoint.

	For each $U_y$ in this finite set, pick $V_x$ so that $V_x \sqcup U_y$ is homeomorphic to $V_x$, and then replace $V_x$ with $V_x \sqcup U_y$. After finitely many steps, the entire end space $E(T)$ is contained in the disjoint union $\bigsqcup_{x \in M} V_x$. See Figure \ref{just_stable_fig} for an example.
\end{proof}

\begin{figure}
	\includegraphics[width=\textwidth]{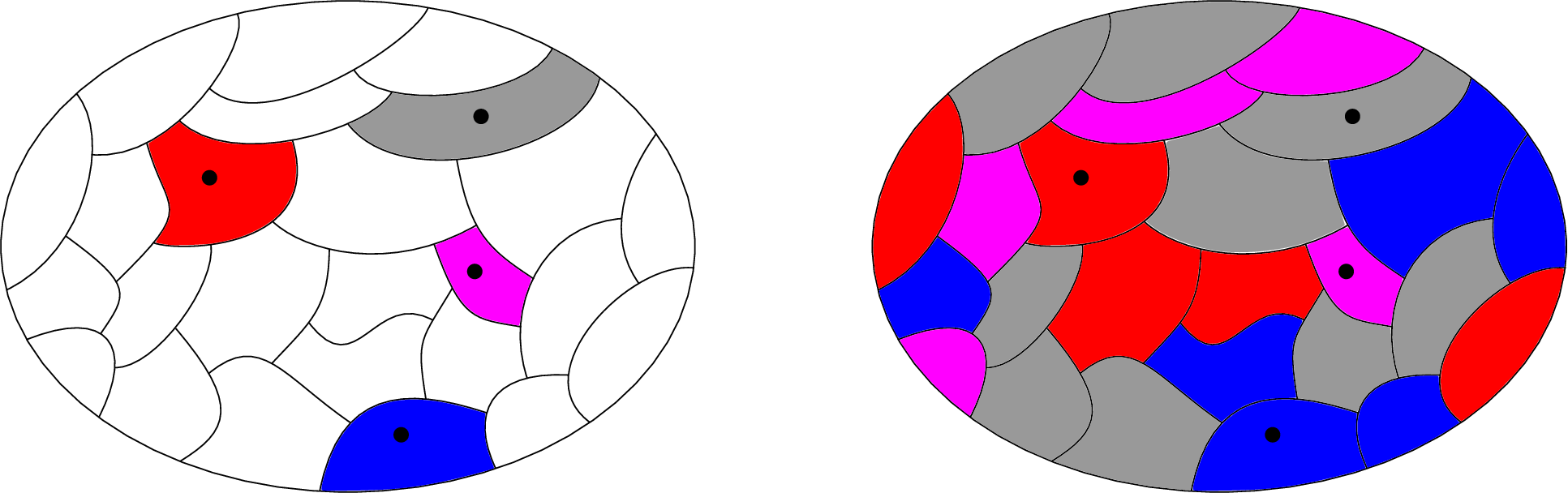}
	\caption{From the proof of Lemma \ref{just_stable}: on the left, the end space $E(T)$ is divided into finitely many disjoint regions by compactness. Each shaded region is a stable neighborhood of a maximal point, shown as a black dot. Each unshaded region satisfies Fact \ref{add_in} with respect to one of the shaded regions. On the right, the shaded regions have been expanded by repeated application of Fact \ref{add_in} so that they now cover the whole surface. Note that each shaded region is still a stable neighborhood of the maximal point shown as a dot; in fact, each large shaded region on the right is homeomorphic to the similarly-colored small region on the left.\label{just_stable_fig}}
\end{figure}

We are now ready to define the canonical collection $\mc{S}$ of subsurfaces that will be used in the definition of the graph $\scg(\Sigma, \mc{S})$. The following construction assumes $\Sigma$ is tame. Let $\set{f_1, \dotsc, f_n, c_1, \dotsc, c_m}$ be representatives of the equivalence classes of maximal ends of $S$, with each $E(f_i)$ intersecting the end space of $S$ finitely many times, and each $E(c_i)$ intersecting the end space of $S$ in a Cantor set. Pick disjoint stable neighborhoods $V_{f_i}$ and $V_{c_i}$ of each representative.

For each $1 \le i \le n$, let $T_i$ be a surface with two boundary components and end space homeomorphic to $V_{f_i} \sqcup \bigsqcup_{j=1}^m V_{c_j}$. If $f_i$ or any of the $c_j$ is accumulated by genus, then by construction $T_i$ will have infinite genus; if none of them are, we further specify that $T_i$ have genus zero. If none of the $f_i$ or $c_i$ are accumulated by genus, but the surface $\Sigma$ has positive genus---in other words, if $S$ has finite positive genus---then let $T_{n+1}$ be the surface with two boundary components, genus $1$, and end space homeomorphic to $\bigsqcup_{j=1}^m V_{c_j}$. Let $\mc{S} = \set{T_1, \dotsc, T_n, (T_{n+1})}$, including $T_{n+1}$ if it has been defined.

We need to handle an edge case: if $n=0$ and $S$ has $0$ or infinite genus, the above construction will give $\mc{S} = \emptyset$, which is not desirable; so in this case we let $\mc{S} = \set{S}$, noting that $S$ is in fact a surface with two boundary components and end space homeomorphic to $\bigsqcup_{j=1}^m V_{c_j}$ by Lemma \ref{just_stable}. It can be seen without too much trouble that in this case $\scg(\Sigma, \mc{S})$ has diameter $2$; this is consistent with the fact that under these conditions the end space of $\Sigma$ is either self-similar (if the maximal ends are all equivalent to $e_+$ and $e_-$) or \emph{telescoping} with respect to $e_+$ and $e_-$ (if they are predecessors), and so by Proposition 3.5 of \cite{mr} the mapping class group $\MCG(\Sigma)$ is coarsely bounded.

From here on we assume $\mc{S}$ is the set of subsurfaces just constructed, and write $\scg(\Sigma)$ to mean $\scg(\Sigma, \mc{S})$.

Before proving connectedness we introduce the following construction, which allows us to produce subsurfaces of $\Sigma$ with nearly arbitrary genus and end space.

\begin{lemma}\label{construct}
	Let $\Sigma$ be a translatable surface with a curve $\alpha$ separating $e_+$ and $e_-$. Then for any clopen subset $V \subseteq E(\alpha_+)$, there is a curve $\beta$ also separating $e_+$ and $e_-$ and such that $E((\alpha, \beta)) = V$. Furthermore, if no end of $V$ is accumulated by genus, but $e_+$ is, there is for every $n \in \nat$ a choice of $\beta$ such that $(\alpha, \beta)$ has genus $n$.
\end{lemma}
\begin{proof}
	Recall that for any separating curve $\gamma$ on $\Sigma$, the end sets of the two components of $\Sigma \setminus \gamma$ are both clopen subsets of $E(\Sigma)$, and that these clopen subsets form a basis for the topology of $E(\Sigma)$.

	By picking clopen neighborhoods of this kind for each end in $V$ and applying compactness, we can describe $V$ as a disjoint union of clopen sets, each of which is bounded by a curve. These can then be combined so that $V$ is bounded by a single curve, as in Figure~\ref{many_to_one}. Call this curve $\eta$. Draw an arc $\lambda$ connecting the curves $\alpha$ and $\eta$, and let $\beta$ be the curve following along $\alpha$, $\lambda$, and $\eta$ as in Figure~\ref{following_curve}. By construction, $\beta$ separates $e_+$ and $e_-$, and $E((\alpha, \beta)) = V$.

	If no end of $V$ is accumulated by genus but $e_+$ is, then $(\alpha, \beta)$ must have finite genus and $\beta_+$ must have infinite genus. By picking a curve $\zeta$ separating a single handle from the rest of $\Sigma$ and then applying the construction of Figure \ref{following_curve} with $\beta$ replacing $\alpha$ and $\zeta$ replacing $\eta$, we get a new curve $\beta'$ such that $(\alpha, \beta)$ and $(\alpha, \beta')$ have the same end space but genus differing by $1$. Doing this finitely many times lets us achieve arbitrary genus for $(\alpha, \beta)$ in this case.
\end{proof}

\begin{figure}
	\includegraphics[width=\textwidth]{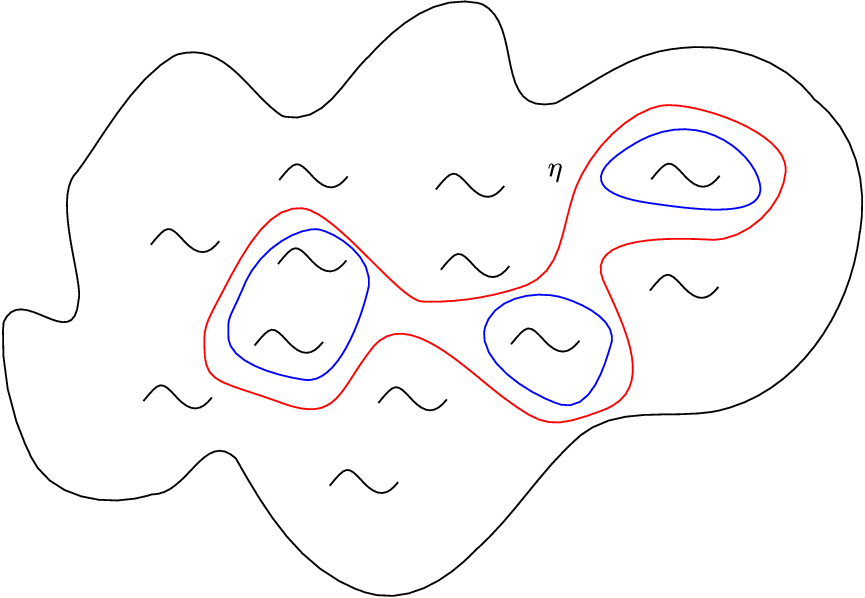}
	\caption{If $V$ is the union of finitely many sets bounded by blue curves, we can find a single red curve $\eta$ bounding $V$. Each squiggle represents a possibly complicated clopen set of ends.\label{many_to_one}}
\end{figure}
\begin{figure}
	\includegraphics[width=\textwidth]{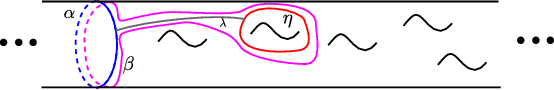}
	\caption{Given $\alpha$ in blue, $\eta$ in red, and $\lambda$ in gray, we draw $\beta$ in magenta. Again, each squiggle represents a possibly complicated clopen set of ends.\label{following_curve}}
\end{figure}

Now that we can construct appropriate subsets, we will be able to build paths between curves in $\scg(\Sigma)$.

\begin{lemma}\label{connected}
	If $\Sigma$ is a translatable surface with tame end space then $\scg(\Sigma)$ is connected.
\end{lemma}
\begin{proof}
	Given $\alpha, \beta \in \scg(\Sigma)$, let $\gamma$ be a curve in $\alpha_+ \cap \beta_+$ such that the subsurfaces $(\alpha, \gamma)$ and $(\beta, \gamma)$ both have end spaces containing representatives of each $E(f_i)$ and $E(c_i)$, using the representatives $\set{f_1, \dotsc, f_n, c_1, \dotsc, c_m}$ defined above; such a curve $\gamma$ can always be found by looking in a small enough neighborhood of $e_+$. We will show that $\alpha$ is connected to $\gamma$; by symmetry, this will imply that $\beta$ is connected to $\gamma$ and so $\alpha$ is connected to $\beta$.

	Let $g$ equal the genus of the subsurface $(\alpha, \gamma)$ if it has finite genus, and $0$ if it has infinite genus. By Lemma \ref{just_stable} the end space of $(\alpha, \gamma)$ can be written as $\bigsqcup_{i=1}^k U_i$ where each $U_i$ is a stable neighborhood of a maximal end. Some of the $U_i$ are stable neighborhoods of an end equivalent to some $f_j$. Let us write these as $\set{X_j}_{j=1}^\ell$, where $\ell$ is a positive integer less than or equal to $k$. The rest of the $U_i$ are stable neighborhoods of some $c_j$. Note in particular that, since $\gamma$ was chosen so that $(\alpha, \gamma)$ contains a representative of each $E(c_i)$, there is at least one $U_j$ that is a stable neighborhood of each $c_j$. Since the disjoint union of two stable neighborhoods of an end equivalent to $c_j$ is itself a stable neighborhood of an end equivalent to $c_j$, we can combine the remaining $U_i$ to get $\set{V_i}_{i=1}^m$, where for each $i$, $V_i$ is a stable neighborhood of an end equivalent to $c_i$. Then $\bigsqcup_{i=1}^k U_i = \big( \bigsqcup_{i=1}^m V_i \big) \sqcup \big( \bigsqcup_{j=1}^\ell X_j \big)$.

	For each $1 \le i \le m$, $E(c_i) \cap V_i$ is a Cantor set, so we can identify $\ell + g$ elements of $E(c_i)$ inside $V_i$ and split $V_i$ into $\bigsqcup_{j=1}^{\ell + g} V_{i,j}$ where each $V_{i,j}$ is a stable neighborhood of an end in $E(c_i) \cap V_i$ and is thus homeomorphic to $V_i$. Then for each $1 \le j \le \ell$, let $W_j = X_j \sqcup \bigsqcup_{i=1}^m V_{i,j}$, and for $\ell + 1 \le j \le \ell + g$ let $W_j = \bigsqcup_{i=1}^m V_{i,j}$. By construction, the end space of $(\alpha, \gamma)$ is $\bigsqcup_{j=1}^{\ell+g} W_j$.

	Finally, define $\set{\alpha_0 = \alpha, \alpha_1, \dotsc, \alpha_{\ell + g} = \gamma}$ such that the end space of $(\alpha_{j-1}, \alpha_j)$ is $W_j$ and such that the genus of $(\alpha_{j-1}, \alpha_j)$ is $0$ or infinite for $j \le \ell$ and $1$ for $j \ge \ell$, which is possible by Lemma \ref{construct}. Then each $(\alpha_{j-1}, \alpha_j)$ is homeomorphic to one of the subsurfaces $T_i \in \mc{S}$, and so $\alpha_{j-1}$ and $\alpha_j$ are adjacent in $\scg(\Sigma)$.
\end{proof}

This leaves us quite close to fulfilling all the hypotheses of Lemma \ref{ms}. We need one more ingredient:

\begin{lemma}\label{dist_one}
	Let $\Sigma$ be a translatable surface with tame end space. Then for any vertex $\alpha$ of $\scg(\Sigma)$, the set $H = \set{f \in \MCG(\Sigma) \mid d(f(\alpha), \alpha) \le 1}$ is coarsely bounded.
\end{lemma}
\begin{proof}
	We will start by defining some helpful mapping classes. Using Proposition \ref{segments}, let $\Sigma = S^{\natural \zah} = \bigntrl_{j \in \zah} S_j$ where all the $S_j$ are homeomorphic and $S_0 = (\alpha, h^N(\alpha))$ for some $N$. As above let $\set{f_1, \dotsc, f_n, c_1, \dotsc, c_m}$ be representatives of the equivalence classes of maximal ends of $S_0$. We may choose these so that each $f_i$ and $c_i$ is actually an end of $S_0$ itself. Let $\set{V_{f_1}, \dotsc, V_{f_n}, V_{c_1}, \dotsc, V_{c_m}}$ be a set of disjoint stable neighborhoods of these ends, also contained in the end space of $S_0$. For each $x$ equal to some $f_i$ or $c_i$ and each $j \in \zah$, let $V_{x,j}$ be a homeomorphic copy of $V_x$ in the end space of the subsurface $S_j$; for instance, we can let let $V_{x, j} = h^{jN} V_x$.

	Note that the sequence of sets $V_{{f_i},j}$ converges to $e_\pm$ as $j$ goes to $\infty_\pm$, and likewise for $V_{{c_i},j}$. That means that for each maximal end $x$ equal to some $f_i$ or $c_i$ there is a homeomorphism of the end space of $\Sigma$ taking each $V_{x,j}$ to $V_{x, j+1}$ and fixing the rest of the end space pointwise. This homeomorphism of the end space extends to a homeomorphism $h_x$ of $\Sigma$; if the end $x$ is not accumulated by genus we may also construct $h_x$ so that $(\alpha, h_x(\alpha))$ has genus $0$. For each $1 \le i \le n$ let $h_i = h_{f_i} \circ \prod_{k=1}^m h_{c_k}$. If $S_0$ has finite positive genus, let $h_\text{genus}$ be a map that moves a single handle from each $S_j$ to $S_{j+1}$ and fixes the end space of $\Sigma$, and then let $h_{n+1} = h_\text{genus} \circ \prod_{k=1}^m h_{c_k}$.

	Observe that, for $1 \le i \le n$, the end space of the surface $(\alpha, h_i(\alpha))$ is homeomorphic to $V_{f_i} \sqcup \bigsqcup_{i=1}^m V_{c_m}$, and has either $0$ or infinite genus depending on whether $f_i$ or any of the $c_i$ is accumulated by genus. Thus $(\alpha, h_i(\alpha))$ is homeomorphic to the subsurface $T_i \in \mc{S}$, where $\mc{S}$ is the canonical set of subsurfaces used to define $\scg(\Sigma) = \scg(\Sigma, \mc{S})$. The same is true of $(h^{-1}_i(\alpha), \alpha)$. If $S_0$ has finite positive genus, then $(\alpha, h_{n+1}(\alpha)$ has genus $1$ and end space homeomorphic to $\bigsqcup_{i=1}^m V_{c_m}$, so $(\alpha, h_{n+1}(\alpha))$ is homeomorphic to $T_{n+1} \in \mc{S}$. Again, the same is true of $(h_{n+1}^{-1}(\alpha),\alpha)$.

	We are now ready to prove that $H$ is coarsely bounded using Fact \ref{cb_nbhd}. Let $V$ be an identity neighborhood in $\MCG(\Sigma)$, and find $n \in \nat$ as in Lemma \ref{shift_enough} so that $V_{\alpha_-} \subseteq h^{-n}Vh^n$ and $V_{\alpha_+} \subseteq h^nVh^{-n}$. Let $F = \set{r^{-1}, h^{\pm n}, h_1^{\pm 1}, \dotsc, h_n^{\pm 1}, (h_{n+1}^{\pm 1})}$, where $r$ is the mapping class defined in Lemma \ref{rotation} and including the maps $h_{n+1}^{\pm 1}$ if they are defined. We claim that $H \subseteq (FV)^8$.

	Let $f \in H$. If $d(\alpha, f(\alpha)) = 0$, then $f \in (FV)^5$ as shown in Corollary \ref{stabilizer}. If not, then $d(\alpha, f(\alpha)) = 1$ and so $\alpha$ and $f(\alpha)$ have disjoint representatives. Assume $f(\alpha) \subseteq \alpha_+$---if not, we will merely have to reverse some signs. By the definition of adjacency in $\scg(\Sigma) = \scg(\Sigma, \mc{S})$, we know that $(\alpha, f(\alpha))$ is homeomorphic to some $T_i \in \mc{S}$. Since $(\alpha, f(\alpha))$ is homeomorphic to $(\alpha, h_i(\alpha))$ by construction, and $f(\alpha)_+$ is homeomorphic to $h_i(\alpha)_+$ by Lemma \ref{transitive}, there is a map $g$ taking $h_i(\alpha)$ to $f(\alpha)$ and restricting to the identity on $\alpha_-$---that is, $g(h_i(\alpha)) = f(\alpha)$ and $g \in V_{\alpha_-} \subseteq (FV)^2$.

	Let $f_0 = h_i^{-1} g^{-1} f$. By construction, $f_0(\alpha) = \alpha$, so by Corollary \ref{stabilizer} $f_0 \in (FV)^5$. Then we have $f = g h_i f_0$, where $g \in (FV)^2$, $h_i \in F$, and $f_0 \in (FV)^5$, so $f \in (FV)^8$.
\end{proof}

Putting this all together gives us

\begin{theorem}\label{shifty_qi}
	If $\Sigma$ is a translatable surface with tame end space, then $\scg(\Sigma)$ equipped with the edge metric is quasi-isometric to $\MCG(\Sigma)$.
\end{theorem}
\begin{proof}
	The group $\MCG(\Sigma)$ is locally coarsely bounded by Corollary \ref{local_cb}. The graph $\scg(\Sigma)$ is connected by Lemma \ref{connected}. The action of $\MCG(\Sigma)$ on it is transitive by Lemma \ref{transitive}. Finally, the set of mapping classes that move a vertex a distance at most $1$ is coarsely bounded by Lemma \ref{dist_one}. Thus by Lemma \ref{ms} the action induces a quasi-isometry.
\end{proof}

\section{Equivalent definitions of translatability}\label{other_graphs}

We have just seen that a translatable surface $\Sigma$ with tame end space is quasi-isometric to the translatable curve graph $\scg(\Sigma)$, which is a graph whose vertices are curves. This section establishes that the existence of such a graph is nearly unique to translatable surfaces.

We say ``nearly unique'' because there is one other example: if $\MCG(\Sigma)$ is coarsely bounded, then it is quasi-isometric to any finite-diameter graph. In particular, the curve graph $\curveg(\Sigma)$ of an infinite-type surface always has diameter $2$, giving a trivial quasi-isometry. For this reason, coarsely bounded mapping class groups are excluded in the hypothesis of Theorem \ref{only_shifty}.

Another condition which we show to be equivalent to translatability is the following, due to Horbez, Qing, and Rafi \cite{hqr}.

\begin{defn}[Definition 1.8 of \cite{mr}]
	A connected, finite-type subsurface $S$ of a surface $\Sigma$ is called \emph{nondisplaceable} if $f(S) \cap S \ne \emptyset$ for each $f \in \MCG(\Sigma)$. A non-connected surface is nondisplaceable if, for every $f \in \MCG(\Sigma)$ and every connected component $S_i$ of $S$, there is a connected component $S_j$ of $S$ such that $f(S_i) \cap S_j \ne \emptyset$.
\end{defn}

\begin{defn}[Definition 4.4 of \cite{hqr}]\label{avenue}
	An \emph{avenue surface} is a connected, orientable surface $\Sigma$ which does not contain any nondisplaceable finite-type subsurfaces, whose end space is tame, and whose mapping class group $\MCG(\Sigma)$ admits a coarsely bounded generating set but is not itself coarsely bounded.
\end{defn}

\begin{theorem}\label{only_shifty}
	Let $\Sigma$ be an infinite-type surface with tame end space such that $\MCG(\Sigma)$ is locally coarsely bounded and admits a coarsely bounded generating set---and thus has a well-defined quasi-isometry type---but is not itself coarsely bounded. Then the following are equivalent:
	\begin{enumerate}
		\item There exists a connected graph $\Gamma$ whose vertices are curves, such that the action of $\MCG(\Sigma)$ on $\Gamma$ is defined and induces a quasi-isometry.
		\item $\Sigma$ is translatable.
		\item $\Sigma$ is an avenue surface.
	\end{enumerate}
\end{theorem}

\begin{proof}
	$2 \implies 1$: This is Theorem \ref{shifty_qi}.

	$2 \implies 3$: We are already assuming that $\Sigma$ has tame end space and that $\MCG(\Sigma)$ admits a coarsely bounded generating set but is not itself coarsely bounded. By the definition of a translation map, a translatable surface cannot have any finite-type nondisplaceable surfaces, and so $\Sigma$ is an avenue surface.

	$3 \implies 2$: Lemma 4.5 of \cite{hqr} says that an avenue surface has zero or infinite genus and exactly two maximal ends, while Lemma 4.6 of \cite{hqr} says that every nonmaximal end of an avenue surface precedes both maximal ends under the standard ordering. It follows by Lemma \ref{almost_two_ends} that $\Sigma$ is translatable.

	$1 \implies 2$: We divide our work into three cases, depending on the genus and maximal ends of $\Sigma$:
	\begin{enumerate}
		\item If $\Sigma$ has zero or infinite genus and one or a Cantor set of equivalent maximal ends, then by Corollary \ref{self_sim} the group $\MCG(\Sigma)$ is coarsely bounded, which is excluded by the hypothesis of the theorem.
		\item If $\Sigma$ has zero or infinite genus and two equivalent maximal ends, then by Proposition \ref{two_ends} it is translatable.
		\item If $\Sigma$ has finite positive genus or any other structure of maximal ends, then by Proposition \ref{no_luck} there is no graph whose vertices are curves and on which the action of $\MCG(\Sigma)$ induces a quasi-isometry, contradicting our assumption.
	\end{enumerate}
	Thus the only remaining possibility is that $\Sigma$ is translatable.
\end{proof}

The following three subsections correspond to the three cases in the last step of the proof of Theorem \ref{only_shifty}.

\subsection{Coarsely bounded mapping class groups}

The first case is essentially a rehash of the following facts. Recall that $\mc{M}(\Sigma)$ is the set of maximal ends of $\Sigma$.

\begin{fact}[Proposition 3.1 of \cite{mr}]
	If $\Sigma$ has zero or infinite genus and self-similar end space, then $\MCG(\Sigma)$ is coarsely bounded.
\end{fact}
\begin{fact}[Proposition 4.8 of \cite{mr}]
	If $\Sigma$ has no nondisplaceable finite-type subsurfaces and $\mc{M}(\Sigma)$ consists of either a singleton or a Cantor set of equivalent ends, then its end space is self-similar.
\end{fact}

To link these two facts together we need to add the assumption of tameness:

\begin{cor}\label{self_sim}
	If $\Sigma$ has zero or infinite genus and tame end space, and $\mc{M}(\Sigma)$ consists of either a singleton or a Cantor set of equivalent ends, then $\MCG(\Sigma)$ is coarsely bounded.
\end{cor}
\begin{proof}
	Given the previous facts, we need only show that $\Sigma$ has no nondisplaceable finite-type subsurfaces. Let $S$ be a finite-type subsurface of $\Sigma$. By expanding $S$ we may assume that $S$ is connected with all its boundary curves essential and separating. Let $E_1, \dotsc, E_n$ be the end spaces of the complementary components of $S$, and $E_0$ the end space of $S$ itself, which may be empty or contain a finite set of punctures. Since $E_0 \sqcup \dotsb \sqcup E_n = E$, there is a maximal end $x$ in some $E_i$; without loss of generality we may assume $x \in E_n$. Let $\alpha$ be the boundary component of $S$ corresponding to $E_n$.

	Since $\Sigma$ has tame end space, $E_n$ contains a stable neighborhood $U$ of $x$. For every end $y \in E \setminus E_n$, $x$ is an accumulation point of $E(y)$, so by Fact \ref{add_in} there is a clopen neighborhood $V_y \subseteq E \setminus E_n$ of $y$ such that $U$ is homeomorphic to $U \sqcup V_y$. By compactness, $E \setminus E_n$ can be covered by finitely many such clopen neighborhoods, meaning that there is a clopen subset $F \subseteq E_n$ homeomorphic to $E \setminus E_n$. Let $\beta$ be a separating curve whose complementary components have end spaces $F$ and $E \setminus F$, and such that the component of $\Sigma \setminus \beta$ with end space $F$ has the same genus as the component of $\Sigma \setminus \alpha$ containing $S$.

	The complementary components of the curves $\alpha$ and $\beta$ have the same genus and end space by construction, so we can find some $f \in \MCG(\Sigma)$ exchanging $\alpha$ and $\beta$. Then $f(S)$ and $S$ are in distinct components of $\Sigma \setminus \alpha$, and so the subsurface $S$ is not nondisplaceable; thus $\Sigma$ has no nondisplaceable finite-type subsurfaces. It follows that the end space of $\Sigma$ is self-similar, and so $\MCG(\Sigma)$ is coarsely bounded.
\end{proof}

\subsection{Translatable surfaces}

\begin{prop}\label{two_ends}
	Suppose $\Sigma$ has tame end space, zero or infinite genus, and exactly two equivalent maximal ends, $e_+$ and $e_-$. If $\MCG(\Sigma)$ is locally coarsely bounded and has a coarsely bounded generating set, then $\Sigma$ is translatable with respect to the ends $e_+$ and $e_-$.
\end{prop}

Our first step towards proving Proposition \ref{two_ends} will be to find the immediate predecessors of the maximal ends of $\Sigma$.

\begin{lemma}\label{pred_exist}
	Let $\Sigma$ be a surface with tame end space and two maximal ends, $e_+$ and $e_-$. If $\MCG(\Sigma)$ is locally coarsely bounded and admits a coarsely bounded generating set, then there is a finite set of ends $x_1, \dotsc, x_n$ such that each $x_i$ is an immediate predecessor of $e_+$, and every end $y \prec e_+$ satisfies $y \preccurlyeq x_i$ for some $i$.
\end{lemma}
\begin{proof}
	Find $K$ as in Fact \ref{cb_class}, with complementary region $A$ containing the end $e_+$, and let $U_0$ be the end space of $A$. Fix $y \prec e_+$; by possibly replacing $y$ with an equivalent end, we may assume $y \in U_0$. Let $U_1$ be a clopen subset of $U_0 \setminus \set{y}$ containing $e_+$, and construct a neighborhood basis $U_0 \supseteq U_1 \supseteq U_2 \supseteq \dotsb$ of clopen sets such that $\bigcap_{n \in \nat} U_n = \set{e_+}$.

	The set $\set{x \in U_0 \setminus U_1 \mid y \preccurlyeq x}$ is nonempty, and is closed by Fact \ref{bigger_closed}, and so by Lemma \ref{closed_max_exist} it has a (not necessarily unique) maximal element which we will call $z_0$. To continue this construction we would like to let $z_n$ be a maximal element of the set $\set{x \in U_n \setminus U_{n+1} \mid z_{n-1} \preccurlyeq x}$ for each $n > 0$, but this set may be empty. Instead we pass to a subsequence, letting $k_0 = 0$ and $k_{n+1}$ be the least natural number greater than $k_n$ such that $(U_{k_n} \setminus U_{k_{n+1}}) \cap E(z_{n-1})$ is nonempty. Then we can set $z_n$ to be a maximal element of $\set{x \in U_{k_n} \setminus U_{k_{n+1}} \mid z_{n-1} \preccurlyeq x}$.

	We claim there is some $z$ such that $z_n \sim z$ for all sufficiently high $n$. If not, then up to taking a subsequence we may assume the $z_n$ are pairwise inequivalent. By construction each $z_n \in U_{k_n}$, so $E(z_n) \cap U_{k_n} \ne \emptyset$. For any $m < n$, $z_m \prec z_n$ but $z_m$ was maximal in $U_{k_m} \setminus U_{k_{m+1}}$, so $E(z_n) \cap (U_{k_m} \setminus U_{k_{m+1}}) = \emptyset$, and thus in general $E(z_n) \cap (U_0 \setminus U_{k_n}) = \emptyset$, or in other words $E(z_n) \subseteq (U_{k_n} \cup U_0^c)$. Finally, let $B$ be a subsurface of $A$ with end space $U_{k_{n+1}}$. By Fact \ref{cb_class} there is a mapping class $f \in \MCG(\Sigma)$ so that $A \subseteq f(B)$. Then $f(z_n) \in U_0^c \cap E(z_n)$ so this set is not empty.

	Let $G = \set{f \in \MCG(\Sigma) \mid f(e_+) = e_+}$. Since the only end of $\Sigma$ that might be equivalent to $e_+$ is $e_-$, $G$ has index at most two in $\MCG(\Sigma)$, and the set $\set{e_+}$ is $G$-invariant. Thus we have fulfilled the definition of limit type, and so by Fact \ref{limit_type} $\MCG(\Sigma)$ cannot admit a coarsely bounded generating set. Since we assumed otherwise, this is a contradiction. This proves our claim that $z_n \sim z$ for all sufficiently high $n$. This $z$ must be an immediate predecessor of $e_+$ (otherwise it would not be maximal in some $U_{k_n} \setminus U_{k_{n+1}}$) and by construction $y \preccurlyeq z$.

	We now claim that there is a clopen subset $F \subseteq U_0$ such that $e_+ \not\in F$ but for every immediate predecessor $z$ of $e_+$, $F \cap E(z) \ne \emptyset$. Suppose not. Then we can pick a sequence of immediate predecessors $\set{z_n}_{n \in \nat}$ of $e_+$ and clopen sets $U_0 \supseteq U_1 \supseteq U_2 \supseteq \dotsb$ with $\bigcap_{n \in \nat} U_n = \set{e_+}$ such that each $z_n \in U_n$ but $E(z_n) \cap (U_0 \setminus U_n) = \emptyset$. As above we can use Fact \ref{cb_class} to find an element of $E(z_n)$ in $U_0^c$, so this would again show that $E$ has limit type, a contradiction by Fact \ref{limit_type}. This proves our claim.

	For each end $y \in F$, let $x_y \in E$ be an immediate predecessor of $e_+$ with $y \preccurlyeq x_y$. That means that for a stable neighborhood $V_{x_y}$ of $x_y$---which must exist because $E$ is tame---there is some clopen neighborhood $U_y$ of $y$ such that $U_y$ is homeomorphic to a clopen subset of $V_{x_y}$. The sets $\set{U_y}_{y \in F}$ cover the compact set $F$, so we can pick a finite collection $U_1, \dotsc, U_n$ with corresponding $x_1, \dotsc, x_n$ predecessors to $e_+$ so that the $U_i$ cover $F$ and each $U_i$ is homeomorphic to a clopen subset of $V_{x_i}$. In particular, by stability $z \preccurlyeq x_i$ for every $z \in U_i$ and so every end in $F$ is bounded above by one of the $x_i$. If $z$ is an immediate predecessor of $e_+$, then by construction there is some $z' \sim z$ in $F$, and so $z \sim z' \preccurlyeq x_i$ for some $i$. Since $z$ is an immediate predecessor of $e_+$, it follows that $z \sim x_i$, so there are only finitely many equivalence classes of immediate predecessors.
\end{proof}

The following lemma is nearly identical to Proposition \ref{two_ends}; we list it separately so that it can be used in other parts of this section.

\begin{lemma}\label{almost_two_ends}
	Let $\Sigma$ be a surface of zero or infinite genus with tame end space and two maximal ends, $e_+$ and $e_-$, with the property that for every end $y \in E \setminus \set{e_+, e_-}$, $y \preccurlyeq e_+$ and $y \preccurlyeq e_-$. If $\MCG(\Sigma)$ is locally coarsely bounded and admits a coarsely bounded generating set, then $\Sigma$ is translatable.
\end{lemma}
\begin{proof}
	First, note that $e_+$ is accumulated by genus if and only if $e_-$ is, by the following argument: suppose $e_+$ is accumulated by genus but $e_-$ is not. If some $y \prec e_+$ were accumulated by genus, then $e_-$ would have to be as well since it is an accumulation point of $E(y)$. So $e_+$ is the only end of $\Sigma$ accumulated by genus. Since $\MCG(\Sigma)$ is locally coarsely bounded, we can find a surface $K$ as in Fact \ref{cb_class}. Let $A$ be the component of $\Sigma \setminus K$ containing $e_+$, and note that $A$ is the only such component with nonzero genus. Pick a subsurface $V \subseteq A$ containing $e_+$ and such that $A \setminus V$ has positive genus. By Fact \ref{cb_class}, there is some $f \in \MCG(\Sigma)$ such that $A \subseteq f(V)$. But that would mean $\Sigma \setminus V$ has positive genus while $\Sigma \setminus f(V)$ has zero genus, a contradiction. Thus $e_+$ is accumulated by genus if and only if $e_-$ is.

	Let $x_1, \dotsc, x_k$ be the immediate predecessors to $e_+$ found via Lemma \ref{pred_exist}. Fix a curve $\alpha_0$ separating $e_+$ and $e_-$. Pick a sequence of pairwise disjoint curves $\alpha_1, \alpha_2, \dotsc$ such that $\lim_{n \to \infty} \alpha_n = e_+$---this is possible by definition for any end---and likewise $\alpha_{-1}, \alpha_{-2}, \dotsc$ such that $\lim_{n \to -\infty} \alpha_n = e_-$. By moving to a subsequence, we may assume that for each $n \in \zah$ the subsurface $(\alpha_n, \alpha_{n+1})$ has positive (possibly infinite) genus if $\Sigma$ has infinite genus, and furthermore that the end space of this subsurface includes an element of each $E(x_i)$.

	By Lemma \ref{just_stable}, write the end space of $(\alpha_n, \alpha_{n+1})$ as $\bigsqcup_{j=1}^p V_{n,j}$, where each $V_{n, j}$ is a stable neighborhood of an immediate predecessor of $e_+$. Note that each $V_{n,j}$ is homeomorphic to some $U_i$, a stable neighborhood of the immediate predecessor $x_i$.

	Thus we can describe the end space of $\Sigma$ as follows: for each $i \in \set{1, \dotsc, k}$, there is a countable collection $\set{U_{i,n}}_{n \in \zah}$ of disjoint homeomorphic copies of $U_i$, whose limit points are $e_+$ and $e_-$ as $n$ approaches $\infty$ and $-\infty$ respectively. That is,
	\[ E = \set{e_-, e_+} \sqcup \bigsqcup_{i=1}^k \left( \bigsqcup_{n \in \zah} U_{i,n} \right) = \set{e_-, e_+} \sqcup \bigsqcup_{n \in \zah} \left( \bigsqcup_{i=1}^k U_{i,n} \right)\]
	where each $U_{i, n}$ is a copy of $U_i$, and the only additional topology is given by the limits
	\[ \lim_{n \to \pm\infty} U_{i, n} = e_{\pm} \]
	for each $1 \le i \le k$.

	Let $S$ be a surface with end space $\bigsqcup_{i=1}^k U_i$ and genus defined as follows: if some $x_i$ is accumulated by genus, $S$ will have infinite genus by definition. If $\Sigma$ has zero genus, let $S$ also have zero genus. If $\Sigma$ has infinite genus but no $x_i$ is accumulated by genus---which implies that only $e_+$ and $e_-$ are accumulated by genus---then let $S$ have genus $1$. Then the surface $S^{\natural \zah}$, which is translatable by construction, has genus and end space matching that of $\Sigma$, and so they are homeomorphic. Thus $\Sigma$ is translatable.
\end{proof}

The proof of Proposition \ref{two_ends} follows directly:

\begin{proof}[Proof of Proposition \ref{two_ends}]
	Since $e_+$ and $e_-$ are the only maximal ends of $\Sigma$, and $e_+ \sim e_-$, every end $y \in E$ has $y \preccurlyeq e_+$ and $y \preccurlyeq e_-$. Then we can apply Lemma \ref{almost_two_ends}.
\end{proof}

\subsection{All other surfaces}\label{neg_results}

\begin{remark}
	Many of the proofs in this subsection are inspired by and to some extent duplicate the work in Mann and Rafi's proof of Fact \ref{cb_class}. They are included for completeness.
\end{remark}

We have covered the cases where $\Sigma$ has zero or infinite genus and either one maximal end, two equivalent maximal ends, or a Cantor set of equivalent maximal ends. We now show that if the maximal ends of $\Sigma$ have any other structure, or if $\Sigma$ has finite positive genus, there is no graph whose vertices are curves onto which the action of $\MCG(\Sigma)$ induces a quasi-isometry. Our main tool will be the following observation of Mann and Rafi:

\begin{fact}[Lemma 5.2 of \cite{mr}]
	Let $K \subseteq \Sigma$ be a finite-type subsurface. If there exists a finite-type, nondisplaceable (possibly disconnected) subsurface $S \subseteq \Sigma \setminus K$, then $V_K$ is not coarsely bounded.
\end{fact}

\begin{cor}\label{nondisp_nograph}
	Let $\Sigma$ be a surface. If for every curve $\alpha$ on $\Sigma$, $\Sigma \setminus \alpha$ contains a finite-type nondisplaceable surface, then there is no graph $\Gamma$ whose vertices are curves on $\Sigma$ such that the orbit map $\MCG(\Sigma) \to \Gamma$ is a quasi-isometry.
\end{cor}
\begin{proof}
	For the orbit map to be a quasi-isometry, the preimage of every bounded set in $\Gamma$ must be coarsely bounded in $\MCG(\Sigma)$. In particular, the stabilizer of a curve is the preimage of a single vertex, so it must be coarsely bounded.
\end{proof}

Mann and Rafi give three basic examples (Examples 2.4 and 2.5 of \cite{mr}) of nondisplaceable surfaces, all of which we will use:
\begin{enumerate}
	\item If $\Sigma$ has finite positive genus, then any subsurface of $\Sigma$ with the same genus as $\Sigma$ is nondisplaceable.
	\item If $X$ is a $\MCG(\Sigma)$-invariant, finite set of ends of $\Sigma$ of cardinality at least $3$, then any surface that separates the elements of $X$ into different complementary components is nondisplaceable.
	\item If $X$ and $Y$ are disjoint, closed $\MCG(\Sigma)$-invariant sets of ends of $\Sigma$ with $X$ homeomorphic to a Cantor set, then a subsurface homeomorphic to a pair of pants containing elements of $X$ in two complementary components, and all of $Y$ in the third, is nondisplaceable.
\end{enumerate}

The easiest place to apply Corollary \ref{nondisp_nograph} is in the case of finite-genus surfaces:

\begin{lemma}\label{finite_genus}
	If $\Sigma$ has finite positive genus, then for any graph $\Gamma$ whose vertices are curves on $\Sigma$, the orbit map $\MCG(\Sigma) \to \Gamma$ is not a quasi-isometry.
\end{lemma}
\begin{proof}
	Fix a curve $\alpha$, and let $S$ be a connected, finite-type subsurface of $\Sigma$ containing $\alpha$ and with the same genus as $\Sigma$. If $\alpha$ is nonseparating in $S$, $S \setminus \alpha$ is still connected and nondisplaceable. If $\alpha$ separates $S$ into two components, one of which has the same genus as $\Sigma$, then that component is connected and nondisplaceable.

	Finally, if $\alpha$ separates $S$ into two components, both of which have positive genus, consider the surface $S \setminus \alpha$. For any $f \in \MCG(\Sigma)$, both components of $f(S \setminus \alpha)$ contain nonseparating curves, and every nonseparating curve on $\Sigma$ intersects $S \setminus \alpha$. Therefore both components of $f(S \setminus \alpha)$ intersect $S \setminus \alpha$, making it a nondisplaceable surface. The result follows by Corollary \ref{nondisp_nograph}.
\end{proof}

Next consider the case where $\Sigma$ has at least three---but finitely many---maximal ends.

\begin{lemma}\label{finite_ends}
	If $\Sigma$ has at least $3$ but finitely many maximal ends, then for any graph $\Gamma$ whose vertices are curves on $\Sigma$, the orbit map $\MCG(\Sigma) \to \Gamma$ is not a quasi-isometry.
\end{lemma}
\begin{proof}
	Fix a curve $\alpha$, and let $S$ be a finite-type surface containing $\alpha$ and separating the maximal ends of $\Sigma$ into distinct complementary components. If $\alpha$ is nonseparating in $S$, then $S \setminus \alpha$ is still connected and nondisplaceable. If $\alpha$ separates $S$ into two components, one of which still separates the maximal ends of $\Sigma$ into distinct complementary components, then that component is connected and nondisplaceable.

	Otherwise, there are at least two maximal ends of $\Sigma$ in both components of $\Sigma \setminus \alpha$. Fix $f \in \MCG(\Sigma)$. Since there are at least two maximal ends of $\Sigma$ in both components of $\Sigma \setminus f(\alpha)$, either $f(\alpha) = \alpha$ or $f(\alpha)$ intersects $S \setminus \alpha$. Since $f(\alpha)$ is a boundary component of both components of $S \setminus \alpha$, it follows that $S \setminus \alpha$ is nondisplaceable. The result follows by Corollary \ref{nondisp_nograph}.
\end{proof}

Now we move to the case of infinitely many maximal ends:

\begin{lemma}\label{inf_ends}
	If $\Sigma$ has infinitely many maximal ends, not all equivalent, then for any graph $\Gamma$ whose vertices are curves on $\Sigma$, the orbit map $\MCG(\Sigma) \to \Gamma$ is not a quasi-isometry.
\end{lemma}
\begin{proof}
	Fix a curve $\alpha$. By Fact \ref{max_exist} the equivalence class of every maximal end is either finite or a Cantor set. If every such equivalence class is finite then there are infinitely many of them; in particular let $x, y, z$ be three nonequivalent maximal ends with $E(x)$, $E(y)$, and $E(z)$ all finite. Then let $X = E(x)$, $Y = E(y)$, and $Z = E(z)$. If on the other hand there is some maximal end $x$ such that $E(x)$ is a Cantor set, pick a maximal end $z$ not equivalent to $x$, and let $X \sqcup Y$ be nonempty sets partitioning $E(x)$, and let $Z = E(z)$.

	In either case above, let $S$ be a finite-type surface containing $\alpha$ and with $X$, $Y$, and $Z$ in distinct complementary components. One component of $S \setminus \alpha$ still has $X$, $Y$, and $Z$ in distinct complementary components, so this component is connected and nondisplaceable. The result follows by Corollary \ref{nondisp_nograph}.
\end{proof}

There is only one more case, which requires a bit more subtlety as well as the condition of tameness:

\begin{lemma}\label{two_diff_ends}
	If $\Sigma$ has tame end space and two non-equivalent maximal ends $e_+$ and $e_-$, and if $\MCG(\Sigma)$ has a well-defined quasi-isometry type, then for any graph $\Gamma$ whose vertices are curves on $\Sigma$, the orbit map $\MCG(\Sigma) \to \Gamma$ is not a quasi-isometry.
\end{lemma}
\begin{proof}
	Since $\MCG(\Sigma)$ has a well-defined quasi-isometry type, we can apply Lemma \ref{pred_exist} to find immediate predecessors to $e_+$ and $e_-$. If $e_+$ and $e_-$ had the same predecessors, $\Sigma$ would be translatable by Lemma \ref{almost_two_ends}, which would imply $e_+ \sim e_-$, a contradiction. Thus without loss of generality we may assume there is some immediate predecessor $x$ of $e_+$ such that $x \not\preccurlyeq e_-$.

	Let $V$ be a stable neighborhood of $x$. Since $x$ is maximal in $V$, $E(x) \cap V$ is either a singleton or a Cantor set. We claim it is in fact a Cantor set. Suppose by contradiction that $E(x) \cap V$ is discrete; since $x$ is an immediate predecessor of $e_+$ and $x \not\preccurlyeq e_-$, this means that $E(x)$ is countable, with a unique accumulation point at $e_+$. Find a subsurface $K$ as in Fact \ref{cb_class}, with complementary components $A_+$ and $A_-$ containing $e_+$ and $e_-$ respectively. Note that all but finitely many elements of $E(x)$ are in the end set of $A_+$. Let $B$ be a subsurface of $A_+$ containing $e_+$, and such that the end space of $A_+ \setminus V$ contains a single element of $E(x)$. Then there is some $f \in \MCG(\Sigma)$ such that $A_+ \subseteq f(B)$. But $\Sigma \setminus A_+$ and $\Sigma \setminus B$ have a different number of elements of $E(x)$. This contradiction proves our claim.

	Since $E(x) \cap V$ is a Cantor set, $x$ is an immediate predecessor of $e_+$, and $x \not\preccurlyeq e_-$, $E(x)$ must be a countable sequence of disjoint Cantor sets converging to $e_+$. Let $X \sqcup Y$ be a partition of $E(x) \cup \set{e_+}$ into nonempty clopen sets, and let $Z = \set{e_-}$. Then a finite-type surface $S$ that has $X$, $Y$, and $Z$ in distinct complementary components will be nondisplaceable. As in the proof of Lemma \ref{inf_ends}, we can construct such an $S$ so that it avoids a fixed curve $\alpha$, and so the result follows by Corollary \ref{nondisp_nograph}.
\end{proof}

These lemmas together give the main result of this subsection:

\begin{prop}\label{no_luck}
	If $\Sigma$ has tame end space and $\MCG(\Sigma)$ has a well-defined quasi-isometry type, and if $\Sigma$ has either finite positive genus, two or infinitely many maximal ends that are not all equivalent, or at least three but finitely many maximal ends, then for any graph $\Gamma$ whose vertices are curves on $\Sigma$, the orbit map $\MCG(\Sigma) \to \Gamma$ is not a quasi-isometry.
\end{prop}
\begin{proof}
	If $\Sigma$ has finite positive genus, this is Lemma \ref{finite_genus}. If it has two maximal ends, this is Lemma \ref{two_diff_ends}. If it has at least three but finitely many maximal ends, this is Lemma \ref{finite_ends}. If it has infinitely many maximal ends, this is Lemma \ref{inf_ends}.
\end{proof}

\section{The plane minus a Cantor set}\label{cantor_sec}

We now turn from the general case of translatable surfaces, of which there are uncountably many examples only a few of which have received specific notice, to a much more specific but more well-studied case. In this section we focus exclusively on the surface $\Sigma = \rea^2 \setminus C$, where $C$ is a Cantor set embedded in the plane. In this instance we will not have to go looking for a suitable graph, as one has been provided for us in the form of the \emph{loop graph} defined by Bavard \cite{bavard}. We will show in this section that the mapping class group of this surface is quasi-isometric to its loop graph.

Note that the surface $\Sigma$ has a unique isolated end, usually called $\infty$ because it is the ``point at infinity'' of $\rea^2$.

\begin{defn}
	A \emph{loop} in $\Sigma$ is an embedded line in $\Sigma$ with both ends approaching $\infty$, considered up to isotopy and orientation reversal. The \emph{loop graph} $\loopg(\Sigma)$ of $\Sigma$ is the graph whose vertices are loops in $\Sigma$, with two loops connected by an edge if they have disjoint representatives.
\end{defn}

It was shown by Bavard \cite{bavard} that the loop graph\footnote{Many papers deal interchangeably with the loop graph as defined here and the \emph{ray graph}, whose vertices are embedded lines with one end at $\infty$ and the other in the Cantor set. These are shown by Bavard \cite{bavard} to be quasi-isometric.} is connected and Gromov-hyperbolic. A subsequent paper of Bavard and Walker \cite{bw} characterized the Gromov boundary of $\loopg(\Sigma)$. The high degree of symmetry possessed by $\Sigma$ also makes the following transitivity lemma possible.

\begin{lemma}\label{cantor_trans}
	If $\alpha$, $\beta$, and $\gamma$ are loops in $\Sigma$, with $\beta$ and $\gamma$ both in the same component of $\Sigma \setminus \alpha$, then there is a mapping class $f \in \MCG(\Sigma)$ such that $f(\alpha) = \alpha$, $f(\beta) = \gamma$, and $f$ restricts to the identity on the component of $\Sigma \setminus \alpha$ not containing $\beta$ and $\gamma$.
\end{lemma}
\begin{proof}
	Observe that every loop on $\Sigma$ is separating. If we cut $\Sigma$ along $\alpha$ and $\beta$, we get three subsurfaces: one whose only boundary component is $\alpha$, one whose only boundary component is $\beta$, and one with both boundary components. The end space of each of these subsurfaces is a nonempty clopen subset of a Cantor set, which must be itself a Cantor set. Since the surface $\Sigma$ has no genus, this is a complete description of the subsurfaces. The same argument applies when cutting the surface along $\alpha$ and $\gamma$, so we can fix the surface bounded by $\alpha$, map the surface bounded by $\beta$ to that bounded by $\gamma$, and map the surface bounded by both $\alpha$ and $\beta$ to that bounded by $\alpha$ and $\gamma$.
\end{proof}

The following lemmas are analogs of Lemmas \ref{rotation} and \ref{shift_enough} in the setting of $\Sigma$:

\begin{lemma}\label{cantor_rot}
	Let $\alpha$ be a loop on $\Sigma$ and $\alpha_-$ and $\alpha_+$ the two components of $\Sigma \setminus \alpha$. Then there is a mapping class $r \in \MCG(\Sigma)$ such that after an isotopy $r(\alpha_+) = \alpha_-$, $r(\alpha_-) = \alpha_+$, and $r_{|\alpha}$ is orientation-reversing.
\end{lemma}
\begin{proof}
	In a tubular neighborhood of $\alpha$, which is a punctured annulus, $r$ is just a rotation by $\pi$ about the line running down the middle of that punctured annulus, as in Figure \ref{rot_fig}. Since $\alpha$ separates the end space of $\Sigma$ into two nonempty clopen sets, the end spaces of $\alpha_-$ and $\alpha_+$ are homeomorphic and so this $r$ can be extended to all of $\Sigma$.
\end{proof}

\begin{lemma}\label{cantor_shift}
	Let $\alpha$ be a loop on $\Sigma$ and $V$ an identity neighborhood in $\MCG(\Sigma)$. Let $\alpha_-$ and $\alpha_+$ be the two components of $\Sigma \setminus \alpha$. Then there are mapping classes $h_+, h_- \in \MCG(\Sigma)$ such that $V_{\alpha_-} \subseteq h_+^{-1} V h_+$ and $V_{\alpha_+} \subseteq h_-^{-1} V h_-$. In addition, $h_+(\alpha) \subseteq \alpha_+$ and $h_-(\alpha) \subseteq \alpha_-$. 
\end{lemma}
\begin{proof}
	Since the sets $\set{V_S \mid \text{$S \subseteq \Sigma$ has finite type}}$ form a neighborhood basis of the identity in $\MCG(\Sigma)$, there is some finite-type $S \subseteq \Sigma$ such that $V_S \subseteq V$. By growing $S$---and thus shrinking $V_S$---we can ensure that the loop $\alpha$ and its basepoint are included in $S$.

	$S$ is a finite-type surface of genus zero, with $n$ boundary components for some $n$. In particular, it must have at least one boundary component in $\alpha_-$ and at least one boundary component in $\alpha_+$. Pick two new arcs $\beta \subseteq \alpha_-$ and $\gamma \subseteq \alpha_+$ and such that both $\beta_- \cap S$ and $\gamma_+ \cap S$ are disks, as in Figure \ref{cantor_shift_fig}. Using Lemma \ref{cantor_trans}, let $h_+, h_- \in \MCG(\Sigma)$ such that $h_+$ fixes $\beta$ and maps $\alpha$ to $\gamma$, while $h_-$ fixes $\gamma$ and maps $\alpha$ to $\beta$.

	It is not quite true that $S \subseteq h_+(\alpha)_- = \gamma_-$ as in the proof of Lemma \ref{shift_enough}. However, the intersection $S \cap \gamma_+$ is a disk, and so any homeomorphism that restricts to the identity on $\gamma_-$ can be homotoped to one restricting to the identity on $S$, and thus $V_{(h_+(\alpha))_-} = V_{\gamma_-} \subseteq V_S \subseteq V$. It follows that $V_{\alpha_-} \subseteq h_+^{-1} V h_+$ and likewise $V_{\alpha_+} \subseteq h_-^{-1} V h_-$.
\end{proof}

\begin{figure}
	\center\includegraphics[width=\textwidth*2/3]{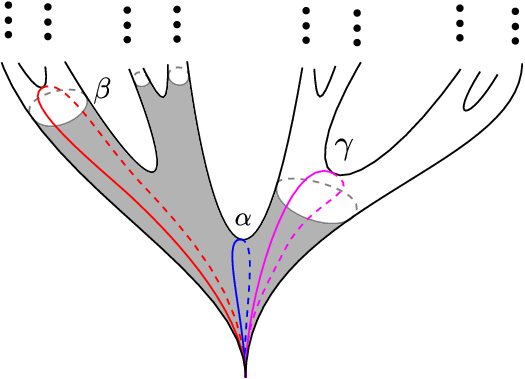}
	\caption{A finite-type subsurface (shaded), with the loop $\alpha$ in blue, the loop $\beta$ in red, and the loop $\gamma$ in magenta.\label{cantor_shift_fig}}
\end{figure}

These are enough ingredients to prove our main theorem for this section:

\begin{theorem}\label{cantor_qi}
	Let $\Sigma = \rea^2 \setminus C$ be the plane minus a Cantor set. Then $\MCG(\Sigma)$ is quasi-isometric to $\loopg(\Sigma)$.
\end{theorem}
\begin{proof}
	The loop graph is known to be connected by work of Bavard \cite{bavard}, and the action of $\MCG(\Sigma)$ on it is transitive by Lemma \ref{cantor_trans}. To apply Lemma \ref{ms} it remains to show that for $\alpha$ a loop on $\Sigma$, the set $A = \set{f \in \MCG(\Sigma) \mid d(\alpha, f(\alpha)) \le 1}$ is coarsely bounded. Fix such an $\alpha$, and refer to the components of $\Sigma \setminus \alpha$ as $\alpha_+$ and $\alpha_-$.

	We will of course be using Fact \ref{cb_nbhd}. Fix an identity neighborhood $V$ in $\MCG(\Sigma)$, and let $r$ and $h$ be as in Lemmas \ref{cantor_rot} and \ref{cantor_shift}. Let $F = \set{r^{-1},\allowbreak h_+,\allowbreak h_-,\allowbreak h_+^{-1},\allowbreak h_-^{-1}}$. We will show that $A \subseteq (FV)^8$.

	Fix $f \in A$. First consider the case where $d(\alpha, f(\alpha)) = 0$. After possibly replacing $f$ with $rf$, we may assume $f$ restricts to the identity on $\alpha$, and so it decomposes as $f = f_- f_+$, where $f_- \in V_{\alpha_-}$ and $f_+ \in V_{\alpha_+}$. Then $f = f_- f_+ \in V_{\alpha_-} V_{\alpha_+} \subseteq h_+^{-1} V h_+ h_-^{-1} V h_- \subseteq (FV)^4$. Since we may have replaced $f$ with $rf$, this gives $f \in (FV)^5$ in general when $d(\alpha, f(\alpha))= 0$.

	Now suppose $d(\alpha, f(\alpha)) = 1$. That means $\alpha$ and $f(\alpha)$ are disjoint. Without loss of generality we assume that $f(\alpha) \subseteq \alpha_+$; if not then we need merely replace $h_+$ with $h_-$ below. By Lemma \ref{cantor_trans} there is some $g \in V_{\alpha_-}$ such that $g(\alpha) = \alpha$ and $g(h_+(\alpha)) = f(\alpha)$. Let $f_0 = h_+^{-1} g^{-1} f$. By construction $f_0(\alpha) = \alpha$ so by the previous paragraph $f_0 \in (FV)^5$. Then $f = g h_+ f_0 \in V_{\alpha_-} h_+ (FV)^5 \subseteq h_+^{-1} V h_+ h_+ (FV)^5 \subseteq (FV)^8$.

	Thus $A \subseteq (FV)^8$, so $A$ is coarsely bounded, and then by Lemma \ref{ms} the action of $\MCG(\Sigma)$ on $\loopg(\Sigma)$ induces a quasi-isometry.
\end{proof}

\subsection{Some consequences of this quasi-isometry}

Theorem \ref{cantor_qi} has some interesting immediate consequences. The first is hyperbolicity; as mentioned in the introduction, $\loopg(\Sigma)$ is known to be $\delta$-hyperbolic.

\begin{cor}\label{hyp}
	Let $\Sigma = \rea^2 \setminus C$. Then $\MCG(\Sigma)$ is non-elementary $\delta$-hyperbolic.
\end{cor}
\begin{proof}
	The mapping class group is quasi-isometric to the loop graph, which was shown by Bavard \cite{bavard} to be non-elementary $\delta$-hyperbolic.
\end{proof}

Since the translatable surfaces are known to have non-hyperbolic mapping class groups, this proves that the mapping class groups are not quasi-isometric.

\begin{cor}\label{distinct}
	The mapping class group of $\rea^2 \setminus C$ is not quasi-isometric to that of any translatable surface.
\end{cor}
\begin{proof}
	The translatable curve graph is never non-elementary hyperbolic by the results of Horbez, Qing, and Rafi \cite{hqr}, and thus neither is the mapping class group of any translatable surface. Thus by Corollary \ref{hyp} the mapping class group of a translatable surface is not quasi-isometric to that of $\rea^2 \setminus C$.
\end{proof}

For the final interesting consequence, we introduce some concepts from the world of locally compact groups. A generating set $S$ for a group $G$ can be thought of as a homomorphism $\phi \colon F_S \to G$ from the free group on the set $S$ to $G$. A collection of words $R \subseteq F_S$ that normally generates the kernel of this map is called a \emph{set of relators} and we often write $G$ as a \emph{group presentation} $G = \langle S \mid R \rangle$. When the sets $S$ and $R$ are both finite, we say the group $G$ is \emph{finitely presented}. Cornulier and de la Harpe \cite{cdlh} introduce the following generalization of this notion.

\begin{defn}
	A group presentation $G = \langle S \mid R \rangle$ is a \emph{bounded presentation} if the words in $R$ have bounded length. In this case we say $G$ is \emph{boundedly presented} over the set $S$.
\end{defn}

Note that a finite presentation is simply a bounded presentation over a finite generating set. Cornulier and de la Harpe call a group \emph{compactly presented} if it has a bounded presentation over a compact generating set, and by analogy we might call a group \emph{coarse-boundedly presented} if it has a bounded presentation over a coarsely bounded generating set. Crucially, Cornulier and de la Harpe show the following close relationship between bounded presentations and word metrics.

\begin{fact}[Proposition 7.B.1 of \cite{cdlh}]\label{rips}
	Let $G$ be a group endowed with a generating set $S$. Then $G$ is boundedly presented over $S$ if and only if the Rips complex $\mathrm{Rips}_c(G, d_S)$ is simply connected for some $c$.
\end{fact}

It follows directly that the mapping class group of the plane minus a Cantor set has a coarsely bounded presentation.

\begin{cor}\label{cbp}
	Let $\Sigma = \rea^2 \setminus C$ be the plane minus a Cantor set. Then $\MCG(\Sigma)$ has a coarsely bounded presentation.
\end{cor}
\begin{proof}
	By Corollary \ref{hyp}, the mapping class group $\MCG(\Sigma)$ is $\delta$-hyperbolic with respect to (any) coarsely bounded generating set $S$. Then by a result credited to Rips by Gromov (Theorem 1.7.A of \cite{gromov}), when $c > 4\delta$ the Rips complex $\mathrm{Rips}_c(G, d_S)$ is contractible, and so by Fact \ref{rips}, the mapping class group $\MCG(\Sigma)$ has a bounded presentation over $S$.
\end{proof}

\bibliographystyle{halpha}
\bibliography{refs}{}

\end{document}